\documentclass{article}

\usepackage[english]{babel}
\usepackage{graphicx}
\usepackage{textcomp}
\usepackage{amssymb}
\usepackage{amsmath}
\usepackage{setspace}
\usepackage[margin=2.5cm]{geometry}
\usepackage{bm}
\usepackage{centernot}

\newtheorem{proposition}{Proposition}
\newtheorem{corollary}{Corollary}
\newtheorem{remark}{Remark}
\newtheorem{definition}{Definition}
\newtheorem{theorem}{Theorem}
\newtheorem{lemma}{Lemma}

\newenvironment{proof}[1][Proof:]{\begin{trivlist}
\item[\hskip \labelsep {\bfseries #1}]}{\end{trivlist}}

\newcommand{\id}{{\mathbf 1}}

\newcommand{\supp}{\operatorname{supp}}

\newcommand{\diag}{\mbox{diag}}

\newcommand{\Hom}{\operatorname{Hom}}

\newcommand{\Aut}{\mbox{\normalfont Aut}}

\newcommand{\ad}{\mbox{\normalfont ad}}

\newcommand{\End}{\mbox{\normalfont End}}

\newcommand{\tr}{\mbox{tr}}

\newcommand{\Der}{\mbox{\normalfont Der}}

\newcommand{\aA}{{{\scriptstyle{}^{\scriptstyle a}}\!\!\mathcal{A}}}

\newcommand{\F}{\mathbb{F}}

\newcommand{\Fc}{{\overline{\mathbb{F}}}}

\newcommand{\Fp}{{\mathbb{F}_{\scriptscriptstyle\!\!/\!p}}}

\newcommand{\Dil}{\mbox{\normalfont Dil}}

\makeatletter
\renewcommand*\env@matrix[1][*\c@MaxMatrixCols c]{%
  \hskip -\arraycolsep
  \let\@ifnextchar\new@ifnextchar
  \array{#1}}
\makeatother

\title{Jordanable almost Abelian Lie algebras}
\author{Zhirayr Avetisyan\\ Department of Mathematics, UCSB}

\begin{document}
\maketitle
\abstract{We call a linear operator on a vector space over a field Jordanable if it has a Jordan canonical form. A Lie algebra will be called Jordanable if its adjoint representation is by Jordanable operators. The existence of canonical form allows to describe important structural properties of these Lie algebras in explicit terms. Lie subalgebras, ideals, automorphisms and derivations, as well as quadratic Casimir elements are given explicitly in a suitable basis.}




\section{Introduction}\label{Intro}

An almost Abelian Lie algebra is a non-Abelian Lie algebra $\mathbf{L}$ over a field $\F$ that contains a codimension one Abelian Lie subalgebra. Such an algebra can always be written as a semidirect product $\F e_0\rtimes\mathbf{V}$ of a one-dimensional Lie algebra $\F e_0$ with a codimension one Abelian ideal $\mathbf{V}$. Among the most prominent representatives of this class are the Heisenberg algebra $\mathbf{H}_\F$, the algebra of generators of affine transformations on the real line $\mathbf{ax+b}_\F$, and the algebra of Killing vector fields on the plane $\mathbf{E}_\F(2)$. In a recent paper \cite{Avetisyan2016} we motivated the introduction of general almost Abelian Lie algebras, studied their structure and obtained a classification up to isomorphism. More precisely, it was shown that these isomorphism classes correspond to the projective similarity classes of the linear operators $\ad_{e_0}\in\End(\mathbf{V})$. In order to obtain more explicit formulae that are important in practice we need to fix a concrete representative from each class, and to work as far as possible with projective similarity invariants of linear operators. This paper is devoted to those almost Abelian Lie algebras for which the operator $\ad_{e_0}$ admits a Jordan canonical form, i.e., is similar to a possibly infinite direct sum of ordinary finite dimensional Jordan blocks. We call such algebras Jordanable and use the convenient canonical forms to solve explicitly some operator equations appearing in the more general results of \cite{Avetisyan2016}. Projective multiplicity functions are introduced as invariant signatures of Jordanable almost Abelian Lie algebras, and all results are expressed in terms of them.

The paper is structured as follows. In Section \ref{FldSpMultF} we establish notation and remind the reader of some elementary facts about fields, polynomials and algebraic extensions. We introduce multiplicity functions and a few operations with them. Section \ref{JordOp} is devoted to Jordanable linear operators on a given vector space. First we discuss conventional Jordan blocks corresponding to various roots $x_*$ and algebraic extensions $\F(x_*)$, and introduce flexible notation for switching between $\F$-linear and $\F(x_*)$-linear objects. A Jordan canonical form $\operatorname{J}(\aleph)$ corresponding to a multiplicity function $\aleph$ is defined as the direct sum of different Jordan blocks with multiplicities given by $\aleph$. Note that no restriction is posed on the cardinality of these multiplicities. A linear operator is defined to be Jordanable if it is similar to a Jordan canonical form. A Jordanable operator can be similar to many Jordan canonical forms, but all these forms differ only by the order of Jordan blocks and have the same multiplicity function, therefore a multiplicity function is a unique similarity invariant in the class of Jordanable operators. The operation of multiplication of a multiplicity function by a scalar introduced in Section \ref{FldSpMultF} is used to establish a bijective correspondence between projective classes of multiplicity functions and projective similarity classes of Jordanable operators.

In Section \ref{InvSubspJordOp} the invariant vector subspaces of a given Jordanable operator are described in terms of a basis adapted to the structure of the multiplicity function. It is shown that the restriction to an invarianr subspace of a Jordanable operator is again Jordanable. The Section \ref{LinEqJordOp} begins with certain technicalities regarding the passage from $\F$-matrices to $\F(x_*)$-matrices. Then we proceed to explicitly describe the solutions of the following operator equations: $\operatorname{X}\operatorname{T}=\lambda\operatorname{T}\operatorname{X}$ and $\operatorname{Y}\operatorname{T}-\operatorname{T}\operatorname{Y}=\operatorname{T}$ for a given $\lambda\in\F^*=\F\setminus\{0\}$ and Jordanable operator $\operatorname{T}$, as well as $\operatorname{Z}\operatorname{J}(\aleph)+\operatorname{J}(\aleph)^\top\operatorname{Z}=0$ for a given nonzero multiplicity function $\aleph$ in the particular (Jordan) basis. Finally, in Section \ref{JordAALieAlg} we introduce Jordanable almost Abelian Lie algebras as those for which the adjoint representation $\ad$ is by Jordanable operators. We show that isomorphism classes of Jordanable almost Abelian Lie algebras correspond to projective classes of multiplicity functions, and for a given multiplicity function $\aleph$ we choose a representative algebra $\aA(\aleph)$. The centre and the lower central series of $\aA(\aleph)$ are described, and it is shown that $\aA(\aleph)$ is nilpotent if and only if $\aleph$ is supported only at the zero root $x_*=0$. It is further shown that a Lie subalgebra of a Jordanable almost Abelian Lie algebra is again Jordanable, and all subalgebras and ideals of $\aA(\aleph)$ are given in terms of $\ad_{e_0}$-invariant subspaces. At the end explicit matrix representations are given for automorphisms $\Aut(\aA(\aleph))$, derivations $\Der(\aA(\aleph))$ and quadratic Casimir elements $Q(x)\in\mathcal{Z}(\mathrm{U}(\aA(\aleph)))$. The terminal Section \ref{Examples} demonstrates the utilization of the above results in studying three examples of concrete almost Abelian Lie algebras: the Bianchi Lie algebra $\mathrm{Bi}(\mathrm{VI}_1)$, the Lie algebra of the Mautner group (the lowest dimensional non-type-I Lie group), and the almost Abelian Lie algebra $\aA_\mathbb{Q}(1\times\sqrt[3]{2}\,^2,q\times0^1)$
over rationals. This choice of examples demonstrates a variety of different structures that trace different trajectories in the flow of arguments leading to the final results.





\section{Fields, spectra and multiplicity functions}\label{FldSpecMultF}

Let us start with establishing notations regarding fields and their spectra. Let $\F$ be a field of scalars, and $\F[X]$ the ring (actually $\F$-algebra) of polynomials $P(X)$ with coefficients in $\F$. For each polynomial $P=P(X)\in\F[X]$ we denote by $\deg P\in\mathbb{N}_0$ its degree. Let $\sigma_\F\subset\F[X]$ be the set of monic irreducible polynomials which is called the \textit{spectrum} of $\F$. It has a natural grading according to the degrees,
$$
\sigma_\F=\bigsqcup_{n\in\mathbb{N}}\sigma_{\F,n},\quad\sigma_{\F,n}\doteq\left\{p\in\sigma_\F|\quad\deg p=n\right\},\quad\forall n\in\mathbb{N}.
$$
For a $p\in\sigma_\F$ there two possibilities. Either $\deg p=1$ and then $p(X)=X-\mu$ for some $\mu\in\F$, or $\deg p>1$ and the equation $p(x_*)=0$ has no roots in $\F$. In the latter case the roots $x_1,...,x_{\deg p}$ are in the algebraic closure $\Fc$ of $\F$. In both cases for every $n\in\mathbb{N}$ the map
$$
\sigma_{\F,n}\ni p(X)=(X-x_1)...(X-x_n)\mapsto\{x_1,...,x_n\}\in\Fc^n/\mathfrak{S}_n,\quad n\in\mathbb{N}
$$
is injective and can be used to embed $\sigma_{\F,n}$ into $\Fc^n/\mathfrak{S}_n$ (here $\mathfrak{S}_n$ is the symmetric group on $n$ symbols).

Let $\F^*\doteq\F\setminus\{0\}$ be the group of invertible elements of $\F$. Define the following action of $\F^*$ on $\F[X]$,
$$
[\lambda\star P](X)=\lambda^{\deg P}P(\frac{X}{\lambda}),\quad\forall\lambda\in\F^*.
$$
\begin{remark}\label{DilInvRemark}The following properties of this action are easily verified,
$$
P(X)=\sum_{k=0}^{\deg P}a_kX^k\quad\Leftrightarrow\quad\lambda\star P(X)=\sum_{k=0}^{\deg P}\lambda^{\deg P-k}a_kX^k,
$$
$$
P(x_*)=0\quad\Leftrightarrow\quad\lambda\star P(\lambda x_*)=0,
$$
$$
\lambda\star(P(X)Q(X))=\lambda\star P(X)\lambda\star Q(X),\quad\forall P(X),Q(X)\in\F[X],\quad\forall\lambda\in\F^*.
$$
\end{remark}
\begin{proposition} For every $\lambda\in\F^*$ and $n\in\mathbb{N}$ we have $\lambda\star\sigma_{\F,n}\subset\sigma_{\F,n}$. More precisely, $\F^*$ acts on $\sigma_{\F,n}\subset\Fc^n/\mathfrak{S}_n$ according to its natural action on $\Fc^n$ by dilations, $\lambda\{x_1,...,x_n\}=\{\lambda x_1,...,\lambda x_n\}$.
\end{proposition}
\begin{proof} That $\F^*$ acts by dilations of the set of roots $\{x_1,...,x_n\}$ is already clear from Remark \ref{DilInvRemark}. We only need to show that $p\in\sigma_{\F,n}$ entails $\lambda\star p\in\sigma_{\F,n}$. That $\deg\lambda\star p=n$ and that $\lambda\star p$ is monic follows again from Remark \ref{DilInvRemark}. Suppose that $\lambda\star p(X)=Q(X)R(X)$ for some $Q,R\in\F[X]$ with $\deg Q,\deg R>0$, i.e., $\lambda\star p$ is reducible. $\star$-multiplying both sides by $\lambda^{-1}$ we get $p(X)=\lambda^{-1}\star Q(X)\lambda^{-1}\star R(X)$ by Remark \ref{DilInvRemark} which contradicts the irreducibility of $p$. $\Box$
\end{proof}
For a subset $\sigma\subset\sigma_\F$ the action of $\F^*$ is defined as pointwise, and the orbit
$$
\F^*\star\sigma\doteq\left\{\sigma'\subset\sigma_\F|\quad\sigma'=\lambda\star\sigma,\quad\lambda\in\F^*\right\}
$$
is the set of all possible subsets of the same cardinality related to $\sigma$ by a uniform dilation.

We proceed to introduce multiplicity functions. Let $\mathcal{C}$ be the class of cardinals as identified with their representative sets up to isomorphism.
\begin{definition} An $\mathbb{N}$-graded multiplicity function is a map $\aleph:\sigma_\F\times\mathbb{N}\mapsto\mathcal{C}$.
\end{definition}
Denote by
$$
\supp\aleph\doteq\left\{p\in\sigma_\F|\quad\aleph(p,\cdot)\not\equiv0\right\}
$$
the \textit{support} of $\aleph$ and by
$$
\dim_\F\aleph\doteq\sum_{p\in\sigma_\F}\sum_{n=1}^\infty n\aleph(p,n)\deg p=\sum_{p\in\supp\aleph}\sum_{n=1}^\infty n\aleph(p,n)\deg p
$$
its \textit{dimension}. The action of $\F^*$ on multiplicity functions can be defined by the pushforward,
$$
[\lambda\star\aleph](\lambda\star p,n)=\aleph(p,n),\quad\forall p\in\sigma_\F,\quad\forall n\in\mathbb{N},\quad\forall\lambda\in\F^*.
$$
\begin{proposition} The $\F^*$-action on multiplicity functions $\aleph$ has the following properties,
$$
\supp\lambda\star\aleph=\lambda\star\supp\aleph,\quad\dim_\F\lambda\star\aleph=\dim_\F\aleph,\quad\forall\lambda\in\F^*.
$$
\end{proposition}
\begin{proof} To justify the first statement write
$$
\supp\lambda\star\aleph=\left\{p\in\sigma_\F|\quad[\lambda\star\aleph](p,\cdot)\not\equiv0\right\}=\left\{\lambda\star p|\quad[\lambda\star\aleph](\lambda\star p,\cdot)\not\equiv0,\quad p\in\sigma_\F\right\}=
$$
$$
\left\{\lambda\star p|\quad\aleph(p,\cdot)\not\equiv0,\quad p\in\sigma_\F\right\}=\left\{\lambda\star p|\quad p\in\supp\aleph\right\}=\lambda\star\supp\aleph.
$$
For the second statement we observe that
$$
\dim_\F\lambda\star\aleph=\sum_{p\in\sigma_\F}\sum_{n=1}^\infty n[\lambda\star\aleph](p,n)\deg p=\sum_{p\in\sigma_\F}\sum_{n=1}^\infty n[\lambda\star\aleph](\lambda\star p,n)\deg\lambda\star p=
$$
$$
\sum_{p\in\sigma_\F}\sum_{n=1}^\infty n\aleph(p,n)\deg\lambda\star p=\dim_\F\aleph,
$$
where we used $\deg\lambda\star p=\deg p$ from Remark \ref{DilInvRemark} in the last step. $\Box$
\end{proof}
We denote the orbit of the $\F^*$-action on $\aleph$ by $\F^*\star\aleph$ and call it a \textit{projective} multiplicity function. The action of $\F^*$ on the class of all multiplicity functions $\mathcal{C}^{\sigma_\F\times\mathbb{N}}$ is not regular in the sense that different orbits may have different ranks. In other words, the isotropy subgroups of $\F^*$ for different multiplicity functions may be different.
\begin{definition} A scalar $\lambda\in\F^*$ will be called a dilation symmetry of the multiplicity function $\aleph$ if $\lambda\star\aleph=\aleph$.
\end{definition}
\begin{remark} It is clear that the set $\Dil(\aleph)$ of all dilation symmetries of a given multiplicity function is the isotropy subgroup of $\F^*$ at $\aleph$.
\end{remark}

In low dimensions when there are finitely many nonzero entries $\{\aleph(p_i,n_i)\}_{i=1}^N$, $N\in\mathbb{N}$ in $\aleph$ the following notation may be more convenient in practice,
\begin{equation}
\aleph\leftrightarrow(m_1\times p_1^{n_1},...,m_N\times p_N^{n_N}),\quad m_1=\aleph(p_1,n_1),...,m_N=\aleph(p_N,n_N).\label{AlephShort}
\end{equation}





\section{Jordanable operators and their multiplicity functions}\label{JordOp}

In this section we are interested in linear operators $\mbox{T}$ on $\F$-vector space $\mathbf{V}$ which are Jordanable, i.e., admit a Jordan canonical form. We refer to Chapter XIV of \cite{Lang2002} for the relevant background and for the theory of finite dimensional operators. In particular, Corollary 2.5 states that every finite dimensional linear operator is Jordanable. The situation is not as clear in infinite dimensions, and below we will try to recover some of the familiar facts in this new realm.

For a given $p\in\sigma_\F$ let $x_p$ be any one root of $p$. Denote by $\Fp\doteq\F(x_p)$ the extension field, which regardless of the choice of the root $x_p$ is isomorphic to $\F[X]/p(X)\F[X]$. Then $[\Fp:\F]=\dim_\F\Fp=\deg p$ and $\Fp\simeq\F^{\deg p}$ as an $\F$-vector space. The left action of $\Fp$ on itself defines the embedding $\Fp\subset\End_\F(\Fp)\simeq\End_\F(\F^{\deg p})$, and the element $x_p\in\Fp$ regarded as an $\F$-matrix takes the form
\begin{equation}
x_p=\begin{pmatrix}
0 & 0 & \ldots & 0 & -a_0\\
1 & 0 & \ldots & 0 & -a_1\\
0 & 1 & \ldots & 0 & -a_2\\
\ldots & \ldots & \ldots & \ldots & \ldots\\
0 & 0 & \ldots & 1 & -a_{\deg p-1}
\end{pmatrix},\label{xpMatrixForm}
\end{equation}
which can be found on page 557 of \cite{Lang2002}. However, in the special case $\F=\mathbb{R}$ it is customary to adopt the form
\begin{equation}
x_p=\begin{pmatrix}
a & -b\\
b & a
\end{pmatrix},\quad p(X)=X^2+a_1X+a_0=(X-a)^2+b^2,\quad b\ge0\label{xpMatrixFormR}
\end{equation}
($b\ge0$ means that of the two complex roots we choose the one in the upper half plane), and in order to remain in contact with the literature on real Jordan forms we will consider this representation as well. We will choose (\ref{xpMatrixForm}) in general and will comment on those results in the sequel which would be affected had we given preference to (\ref{xpMatrixFormR}). In order to distinguish between two conventions we introduce the indicator
\begin{equation}
\epsilon=\begin{cases}
1\quad\mbox{if (\ref{xpMatrixForm}) is adopted,}\\
0\quad\mbox{if $\F=\mathbb{R}$ and (\ref{xpMatrixFormR}) is adopted.}
\end{cases}\label{epsilonDef}
\end{equation}
In both cases, for $\deg p=1$ when $p(X)=X+a_0$ we have $x_p=-a_0$.

Given an $\Fp$-vector space $\mathbf{V}$ we can consider it as an $\F$-vector space $\mathbf{V}\otimes_\F\Fp$ (tensor product of $\F$-vector spaces), so that $\Fp$-matrices $\operatorname{T}\in\End_\Fp(\mathbf{V})$ become $\F$-block-matrices $\operatorname{T}\in\End_\F(\mathbf{V}\otimes_\F\Fp)$ with $\Fp\subset\End_\F(\Fp)$-valued blocks. Hereafter we will freely switch between $\Fp$-matrices and $\F$-block matrices emphasizing the ground field whenever not obvious from the context. For instance, let us for $n\in\mathbb{N}$ denote by $\id_n$ the $n\times n$ identity matrix, where the ground field will be explicitly specified or clear from the context. If the field is $\Fp$ then as an $\Fp$-matrix $\id_n$ is $n\times n$, whereas as an $\F$-matrix it is the $n\deg p\times n\deg p$ matrix $\id_n\otimes\id_{\deg p}$. More generally, by tensor product $\operatorname{A}\otimes\operatorname{B}$ we will mean the block matrix obtained by multiplying blocks $\operatorname{B}$ with the entries of the $n\times m$ matrix $\operatorname{A}$,
\begin{equation}
\operatorname{A}\otimes\operatorname{B}=\begin{pmatrix}
A_{1;1}\operatorname{B} & A_{1;2}\operatorname{B} & \ldots & A_{1;m}\operatorname{B}\\
\ldots & \ldots & \ldots & \ldots\\
A_{n;1}\operatorname{B} & A_{n;2}\operatorname{B} & \ldots & A_{n;m}\operatorname{B}
\end{pmatrix}.\label{MatrixTensorProd}
\end{equation}

For every $p\in\sigma_\F$ and $n\in\mathbb{N}$ let $\mbox{\normalfont J}(p,n)$ be the $\Fp$-linear operator (Jordan block) on the vector space (block space) $\Fp^n$ given by the matrix form
\begin{equation}
\operatorname{J}(p,n)=x_p\id_n+\mathrm{N}_n\in\End_\Fp(\Fp^n),\label{JpnDef}
\end{equation}
where $x_p\in\Fp$ is as in (\ref{xpMatrixForm}) and $\mathrm{N}_n$ is the nilpotent $\Fp$-matrix
$$
\mathrm{N}_n=\begin{pmatrix}
0 && 1 && 0 && ... && 0\\
0 && 0 && 1 && ... && 0\\
&& && ... && &&\\
0 && 0 && 0 && ... && 1\\
0 && 0 && 0 && ... && 0
\end{pmatrix}
$$
(note that our convention is the transpose of Theorem 2.4 in Chapter XIV of \cite{Lang2002}). 
Each such block $\Fp^n$ is an indecomposable $\Fp[\operatorname{J}(p,n)]$-module, i.e., it cannot be written as a direct sum of two $\operatorname{J}(p,n)$-invariant $\Fp$-vector subspaces.

Let $\operatorname{T}_i\in\End_\F(\mathbf{V}_i)$ for $i=1,2$. The operators $\operatorname{T}_1$ and $\operatorname{T}_2$ are called \textit{similar} (written $\operatorname{T}_1\sim\operatorname{T}_2$) if there exists an invertible linear operator (intertwiner) $\operatorname{S}:\mathbf{V}_1\to\mathbf{V}_2$ such that $\operatorname{S}\operatorname{T}_1=\operatorname{T}_2\operatorname{S}$. The two operators are called \textit{projectively} similar if $\lambda\operatorname{T}_1\sim\operatorname{T}_2$ for some $\lambda\in\F^*$.

\begin{proposition}\label{DilSimProp} For every $p,q\in\sigma_\F$, $m,n\in\mathbb{N}$ and $\lambda\in\F^*$, $\lambda\operatorname{J}(p,m)\sim\operatorname{J}(q,n)$ as $\F$-matrices if and only if $q=\lambda\star p$ and $n=m$.
\end{proposition}
\begin{proof} First assume $\lambda\operatorname{J}(p,m)\sim\operatorname{J}(q,n)$ as $\F$-matrices. Then 
$$
[\lambda\star p]^m(\operatorname{J}(q,n))\sim[\lambda\star p]^m(\lambda\operatorname{J}(p,m))=p^m(\operatorname{J}(p,m))=0.
$$
But also $q^n(\operatorname{J}(q,n))=0$. If $\lambda\star p\neq q$ then $[\lambda\star p]^m$ and $q^n$ are mutually prime, and by Bezout theorem it follows that $\operatorname{J}(q,n)=0$ which is not true. Therefore $\lambda\star p=q$ and by Remark \ref{DilInvRemark} $\deg q=\deg p$. But similar matrices have the same dimensions, hence $m\deg p=n\deg q$ and thus $m=n$. Conversely, $\F[\lambda\operatorname{J}(p,m)]=\F[\operatorname{J}(p,m)]$ acts indecomposably on $\Fp^m$, therefore the matrix $\lambda\operatorname{J}(p,m)$ has only a single Jordan block, i.e., $\lambda\operatorname{J}(p,n)\sim\operatorname{J}(q,n)$. But then we have already shown that $q=\lambda\star p$ and $n=m$. $\Box$
\end{proof}
Later in Lemma \ref{lambdaJLemma} we will find the corresponding intertwiner explicitly.

\begin{definition} A canonical form $\operatorname{J}(\aleph)\in\End_\F(\F^{\dim_\F\aleph})$ associated with a multiplicity function $\aleph$ is a linear operator with matrix form
\begin{equation}
\operatorname{J}(\aleph)=\bigoplus_{p\in\sigma_\F}\bigoplus_{n=1}^\infty\bigoplus_{\aleph(p,n)}\operatorname{J}(p,n).\label{CanFormDef}
\end{equation}
\end{definition}
Note that this definition is intrinsically ambiguous: with no a priori order on the set $\sigma_\F$, different orders of direct sum will in general yield different matrix forms. One can fix this freedom by introducing an arbitrary order on $\sigma_\F$. But we will not need to do so as we will be interested in canonical forms only up to similarity. Different orders in the direct sum of blocks correspond to different choices of basis in the same vector space. A change of basis is a similarity transformation, and all possible orders of summands fall into the same similarity class. Thus, if we denote by $[\operatorname{T}]$ the similarity class containing the linear operator $\operatorname{T}$, the map $\aleph\mapsto[\operatorname{J}](\aleph)$ from multiplicity functions to similarity classes is well defined.

\begin{definition} A linear operator $\operatorname{T}\in\End_\F(\mathbf{V})$ on an $\F$-vector space $\mathbf{V}$ will be called Jordanable if there exists a multiplicity function $\aleph_{\operatorname{T}}$ such that $\operatorname{T}\sim\operatorname{J}(\aleph_{\operatorname{T}})$.
\end{definition}
In this case $\aleph_{\operatorname{T}}$ is referred to as the multiplicity function of $\operatorname{T}$, and $\operatorname{J}(\aleph_{\operatorname{T}})$ is called a Jordan canonical form of $\operatorname{T}$. Below we make sure that $\aleph_{\operatorname{T}}$ is unique by showing how to extract it from $\operatorname{T}$.

\begin{proposition}\label{MultFromOpProp} For a Jordanable operator $\operatorname{T}$, $\forall p\in\sigma_\F$, $\forall n\in\mathbb{N}$  we have
$$
\aleph_{\operatorname{T}}(p,n)=\dim_\Fp\ker p^n(\operatorname{T})/p(\operatorname{T})\ker p^{n+1}(\operatorname{T})-\dim_\Fp\ker p^{n-1}(\operatorname{T})/p(\operatorname{T})\ker p^n(\operatorname{T}).
$$
\end{proposition}
\begin{proof} Let us first deal with the case $\operatorname{T}=\operatorname{J}(\aleph)$ (we drop the index $\operatorname{T}$ of $\aleph_{\operatorname{T}}$ for brevity). For every fixed $p\in\sigma_\F$ the subspace
$$
\mathbf{V}_p\doteq\sum_{n=1}^\infty\ker p^n(\operatorname{J}(\aleph))=\bigoplus_{n=1}^\infty\bigoplus_{\aleph(p,n)}\Fp^n
$$
is actually an $\Fp$-vector space. Choose an $\Fp$-basis $\{\xi_\alpha^m(p,n)\}$ with $\alpha\in\aleph(p,n)$ and $1\le m\le n$ such that
$$
\mathbf{V}_p=\bigoplus_{n=1}^\infty\bigoplus_{\alpha\in\aleph(p,n)}\bigoplus_{m=1}^n\Fp\xi_\alpha^m(p,n),\quad p(\operatorname{J}(\aleph))\xi_\alpha^m(p,n)=\begin{cases}
\xi_\alpha^{m-1}(p,n),\quad m>1\\
0,\quad m=1
\end{cases}.
$$
Now it is not difficult to see that
$$
\ker p^k(\operatorname{J}(\aleph))=\bigoplus_{n=1}^\infty\bigoplus_{\alpha\in\aleph(p,n)}\bigoplus_{m=1}^{\min\{k,n\}}\Fp\xi_\alpha^m(p,n),
$$
$$
p(\operatorname{J}(\aleph))\ker p^{k+1}(\operatorname{J}(\aleph))=\bigoplus_{n=1}^\infty\bigoplus_{\alpha\in\aleph(p,n)}\bigoplus_{m=1}^{\min\{k+1,n\}-1}\Fp\xi_\alpha^m(p,n),
$$
whence
$$
\ker p^k(\operatorname{J}(\aleph))/p(\operatorname{J}(\aleph))\ker p^{k+1}(\operatorname{J}(\aleph))\simeq\bigoplus_{n=1}^k\bigoplus_{\alpha\in\aleph(p,n)}\Fp\xi_\alpha^n(p,n),
$$
and therefore
$$
\dim_\Fp\ker p^k(\operatorname{J}(\aleph))/p(\operatorname{J}(\aleph))\ker p^{k+1}(\operatorname{J}(\aleph))=\sum_{n=1}^k\aleph(p,n).
$$
Now come back to the general case with $\operatorname{S}$ being the intertwiner so that $\operatorname{S}\operatorname{T}\operatorname{S}^{-1}=\operatorname{J}(\aleph)$. Since $\operatorname{S}$ is an isometry,
$$
\dim_\Fp\ker p^k(\operatorname{T})/p(\operatorname{T})\ker p^{k+1}(\operatorname{T})=\dim_\Fp\operatorname{S}\ker p^k(\operatorname{T})/\operatorname{S}p(\operatorname{T})\ker p^{k+1}(\operatorname{T})
$$
$$
=\dim_\Fp\ker p^k(\operatorname{J}(\aleph))/p(\operatorname{J}(\aleph))\ker p^{k+1}(\operatorname{J}(\aleph))=\sum_{n=1}^k\aleph(p,n).
$$
The result easily follows. $\Box$ 
\end{proof}
\begin{corollary}\label{MultMapInjCor} The map $\aleph\mapsto[\operatorname{J}(\aleph)]$ from multiplicity functions to similarity classes is injective.
\end{corollary}
Finally we show how the multiplicity function captures projective similarity.
\begin{proposition}\label{DilSimMultProp} For two multiplicity functions $\aleph,\aleph'$ and a scalar $\lambda\in\F^*$, $\operatorname{J}(\aleph')\sim\lambda\operatorname{J}(\aleph)$ if and only if $\aleph'=\lambda\star\aleph$.
\end{proposition}
\begin{proof}First let $\aleph'=\lambda\star\aleph$ so that
$$
\operatorname{J}(\aleph')=\operatorname{J}(\lambda\star\aleph)=\bigoplus_{p\in\sigma_\F}\bigoplus_{n=1}^\infty\bigoplus_{\lambda\star\aleph(p,n)}\operatorname{J}(p,n)=\bigoplus_{p\in\sigma_\F}\bigoplus_{n=1}^\infty\bigoplus_{\aleph(p,n)}\operatorname{J}(\lambda\star p,n).
$$
By Proposition \ref{DilSimProp} there exists an invertible intertwiner $\operatorname{S}_{p,n}\in\End_\F(\Fp^n)$ such that $\operatorname{S}_{p,n}\operatorname{J}(\lambda\star p,n)=\lambda\operatorname{J}(p,n)\operatorname{S}_{p,n}$. It follows that $\operatorname{S}\operatorname{J}(\lambda\star\aleph)=\lambda\operatorname{J}(\aleph)\operatorname{S}$, where
$$
\operatorname{S}=\bigoplus_{p\in\sigma_\F}\bigoplus_{n=1}^\infty\bigoplus_{\aleph(p,n)}\operatorname{S}_{p,n}.
$$
For the converse let us assume that $\operatorname{J}(\aleph')\sim\lambda\operatorname{J}(\aleph)$ for some multiplicity functions $\aleph,\aleph'$ and a scalar $\lambda\in\F^*$. Above we have shown that $\lambda\operatorname{J}(\aleph)\sim\operatorname{J}(\lambda\star\aleph)$, so that by transitivity $\operatorname{J}(\aleph')\sim\operatorname{J}(\lambda\star\aleph)$. The assertion follows directly from Corollary \ref{MultMapInjCor}. $\Box$
\end{proof}
\begin{corollary}\label{ProjMultProjSimCor} The map $\F^*\aleph\mapsto[\F^*\operatorname{J}](\aleph)$ from projective multiplicity functions to projective similarity classes of Jordanable operators is also injective.
\end{corollary}

We end this section by putting the question of Jordanability into the context of general commutative algebra, where it truly belongs.
\begin{remark}\label{JordModRemark} The linear operator $\mathrm{T}$ on an $\F$-vector space $\mathbf{V}$ is Jordanable if and only if the $\F[X]$-module $\mathbf{V}$ is a direct sum of finite length submodules.
\end{remark}
Indeed, it is easy to see that a finite length $\F[X]$-submodule of $\mathbf{V}$ is a finite direct sum of Jordan blocks spaces.





\section{Invariant subspaces of a Jordanable operator}\label{InvSubspJordOp}

In this section we will describe the invariant subspaces $\mathbf{W}\subset\mathbf{V}$ of a Jordanable operator $\mathrm{T}$ on a vector space $\mathbf{V}$ over a field $\F$. This will later be used to describe subalgebras and ideals of an almost Abelian Lie algebra. Aside from that, invariant subspaces are of considerable importance in the context of linear dynamical systems. A very detailed exposition of invariant subspaces for linear operators on finite dimensional real and complex vector spaces can be found, for instance, in \cite{Gohberg2006}. Here we will touch the subject very briefly, but in a larger context of arbitrary vector spaces.

In order to do so let us give an adaptation to our setting of a theorem by Kulikov on commutative groups.
\begin{theorem}\label{KulikovTheorem} For $p\in\sigma_\F$, an $\Fp$-module $\mathbf{V}_p$ is a direct sum of cyclic submodules if and only if $\mathbf{V}_p$ is the union of an ascending chain of submodules,
$$
\mathbf{V}_p^1\subset\mathbf{V}_p^2\subset\ldots,\quad\mathbf{V}=\bigcup_{n=1}^\infty\mathbf{V}_p^n,
$$
such that for every $n\in\mathbb{N}$, the heights of nonzero elements in $\mathbf{V}_p^n$ are bounded by natural numbers $k_n\in\mathbb{N}$.
\end{theorem}
The original proof can be found in (\cite{Fuchs1970}, Theorem 17.1 on page 87), and it can be extended to modules over PID with only trivial modifications.

Let $\mathrm{T}$ be a Jordanable operator on the $\F$-vector space $\mathbf{V}$, and let $\mathbf{W}\subset\mathbf{V}$ be a $\mathrm{T}$-invariant vector subspace.
\begin{proposition}\label{JordOpRestrProp} The restriction $\mathrm{T}|_\mathbf{W}$ of the Jordanable linear operator $\mathrm{T}$ to an invariant vector subspace $\mathbf{W}$ is again Jordanable.
\end{proposition}
\begin{proof} Let
$$
\mathbf{V}=\bigoplus_{p\in\sigma_\F}\mathbf{V}_p,\quad\mathbf{W}=\bigoplus_{p\in\sigma_\F}\mathbf{W}_p
$$
be the corresponding spectral decompositions. The $\Fp$-module $\mathbf{V}_p$ is obviously a direct sum of cyclic submodules (Jordan block spaces), so that by Theorem \ref{KulikovTheorem} we obtain a chain of submodules $\{\mathbf{V}_p^n\}_{n=1}^\infty$ with the above mentioned properties. Denote $\mathbf{W}_p^n\doteq\mathbf{V}_p^n\cap\mathbf{W}_p$, which is an $\Fp$-submodule of $\mathbf{W}_p$. The heights of non-zero elements in $\mathbf{W}_p^n\subset\mathbf{V}_p^n$ are obviously bounded by the same $k_n$ as above. Moreover,
$$
\mathbf{W}_p=\mathbf{W}_p\cap\mathbf{V}_p=\mathbf{W}_p\cap\left(\bigcup_{n=1}^\infty\mathbf{V}_p^n\right)=\bigcup_{n=1}^\infty\left(\mathbf{W}_p\cap\mathbf{V}_p^n\right)=\bigcup_{n=1}^\infty\mathbf{W}_p^n.
$$
Thus again by Theorem \ref{KulikovTheorem} we learn that each $\mathbf{W}_p$, and thereby also $\mathbf{W}$ is a direct sum of cyclic $\F[X]$-submodules of $\mathbf{V}$. Every such cyclic submodule is of finite length. Indeed, if $v\in\mathbf{V}$ is a cyclic vector then the submodule equals $\F[\mathrm{T}]v$. Write $v=v_1+\ldots v_k$, where $v_1,\ldots,v_k$ are cyclic vectors in the original Jordan decomposition. Then our cyclic submodule $\F[\mathrm{T}]v$ is a submodule of the direct sum $\F[\mathrm{T}]v_1\oplus\ldots\oplus\F[\mathrm{T}]v_k$, which is of finite length by itself. Therefore by Remark \ref{JordModRemark} $\mathrm{T}|_\mathbf{W}$ is Jordanable. $\Box$
\end{proof}
Thus the description of an invariant subspace of a Jordanable operator reduces to that of an irreducible invariant subspace, i.e., one with a single Jordan block space.
\begin{corollary}\label{InvSubspCorr} An invariant subspace of a Jordanable operator on a vector space $\mathbf{V}$ with Jordan decomposition
$$
\mathbf{V}=\bigoplus_{p\in\sigma_\F}\bigoplus_{n=1}^\infty\bigoplus_{\alpha\in\aleph(p,n)}\bigoplus_{m=1}^n\Fp\xi^m_\alpha(p,n)
$$
is any subspace of the form
\begin{equation}
\mathbf{W}=\bigoplus_{p\in\sigma_\F}\bigoplus_{n=1}^\infty\bigoplus_{\beta\in\beth(p,n)}\bigoplus_{m=1}^n\Fp\eta^m_\beta(p,n),\label{InvSubspForm}
\end{equation}
where $\beth$ is a multiplicity function and
\begin{equation}
\eta^m_\beta(p,n)=\sum_{k=1}^\infty\sum_{\alpha\in\aleph(p,k)}\sum_{l=1}^{\min\{k,m\}}\mu_p(n,\beta;k,\alpha,m-l)\xi^l_\alpha(p,k),\quad\forall p\in\sigma_\F,\quad\forall n\in\mathbb{N},\quad m=1,\ldots,n.\label{InvSubspTransform}
\end{equation}
For every $\beta\in\beth(p,n)$ there exists an $\alpha\in\aleph(p,\bar n)$ with $\bar n\ge n$ such that the $\F$-valued coefficient function $\mu_p(n,\beta;\bar n,\alpha,0)\neq0$. The support of $\mu_p(n,\beta;.,.,.)$ is finite.
\end{corollary}
\begin{proof} In view of Proposition \ref{JordOpRestrProp} it suffices to prove that a subspace of the form
$$
\bigoplus_{m=1}^n\Fp\eta^m_\beta(p,n)
$$
is a Jordan block if and only if (\ref{InvSubspTransform}) holds. Being a Jordan block amounts to $p(\mathrm{T})\eta^m_\beta(p,n)=\eta^{m-1}_\beta(p,n)$ for $1<m\le n$ and $p(\mathrm{T})\eta^1_\beta(p,n)=0$. The rest follows easily. $\Box$
\end{proof}





\section{Linear equations with Jordanable operators}\label{LinEqJordOp}

In \cite{Avetisyan2016} the explicit representations of several properties of an almost Abelian Lie algebra (e.g., automorphisms and derivations) were reduced to the solution of a few linear operator equations. These equations will be solved in this section under the assumption that the coefficient operators are Jordanable. For the case of finite dimensional matrices over $\F\in\{\mathbb{R},\mathbb{C}\}$ the general solution of linear matrix equations can be found in literature (e.g., \cite{Gantmacher1959}), but here we will obtain more explicit answers for the particular cases at hand.

Let us start by establishing the following simple fact.
\begin{lemma}\label{FpSelfNormLem} For every $p\in\sigma_\F$, the subring $\Fp\subset\End_\F(\Fp)$ is self-normalizing, i.e., $[\operatorname{T},\Fp]\subset\Fp$ implies $\operatorname{T}\in\Fp$ for every $\operatorname{T}\in\End_\F(\Fp)$.
\end{lemma}
\begin{proof} If $\deg p=1$ then $\F=\Fp=\End_\F(\Fp)$ and the claim is trivial, so we will concentrate on the case $\deg p>1$. Let us first note that $\Fp\subset\End_\F(\Fp)$ is self-centralizing, that is, $[\operatorname{T},\Fp]=0$ implies $\operatorname{T}\in\Fp$ for every $\operatorname{T}\in\End_\F(\Fp)$. Indeed, since $x=x1$ for every $x\in\Fp$, we have
$$
\operatorname{T}x=\operatorname{T}(x1)=(\operatorname{T}x)1=(x\operatorname{T})1=x(\operatorname{T}1)=x((\operatorname{T}1)1)=(x(\operatorname{T}1))1=((\operatorname{T}1)x)1=(\operatorname{T}1)x,\quad\forall x\in\Fp,
$$
which shows that $\operatorname{T}=\operatorname{T}1\in\Fp$. Now let $\Fp=\F(x_p)$ and let by the assumption $[\operatorname{T},x_p]=y$ for some $y\in\Fp$. This can be written as $\operatorname{T}x_p=x_p\mbox{T}+y$. One can prove inductively that
$$
\operatorname{T}x_p^m=x_p^m\operatorname{T}+mx_p^{m-1}y,
$$
which in its turn implies
$$
\operatorname{T}P(x_p)=P(x_p)\operatorname{T}+P'(x_p)y,\quad\forall P(X)\in\F[X].
$$
Choosing $P(X)=p(X)$ so that $p(x_p)=0$ gives $p'(x_p)y=0$. But $p(X)$ is the minimal polynomial of $x_p$ and $\deg p'<\deg p$, therefore $p'(x_p)\neq0$. This shows that $y=0$, and thus $[\operatorname{T},\Fp]=[\operatorname{T},\F(x_p)]=0$, which by the self-centralization of $\Fp$ proves that $\operatorname{T}\in\Fp$. $\Box$
\end{proof}
For every $m,n\in\mathbb{N}$, $p\in\sigma_\F$ and rectangular matrix $\operatorname{T}\in\Hom_\F(\Fp^m,\Fp^n)$ denote
$$
[\operatorname{T},x_p\id]=\operatorname{T}(\id_m\otimes x_p)-(\id_n\otimes x_p)\operatorname{T},
$$
were the tensor product is understood in the same sense as in (\ref{MatrixTensorProd}).
\begin{corollary}\label{xpCommFptoFpCor} For every $p\in\sigma_\F$, $m,n\in\mathbb{N}$, and for every rectangular matrix $\operatorname{T}\in\Hom_\F(\Fp^m,\Fp^n)$, the statement $[\operatorname{T},x_p\id]\in\Hom_\Fp(\Fp^m,\Fp^n)$ implies $\operatorname{T}\in\Hom_\Fp(\Fp^m,\Fp^n)$.
\end{corollary}
\begin{proof}Every matrix $\operatorname{T}\in\Hom_\F(\Fp^m,\Fp^n)$ can be considered as a block matrix of dimension $n\times m$ with blocks being elements of $\End_\F(\Fp)$. Therefore the commutator
$$
\operatorname{R}\doteq\operatorname{T}(\id_m\otimes x_p)-(\id_n\otimes x_p)\operatorname{T}\in\Hom_\F(\Fp^m,\Fp^n)
$$
is a block matrix with blocks being $\operatorname{R}_{k,l}=[\operatorname{T}_{k,l},x_p]\in\End_\F(\Fp)$, $k=1,...,n$, $l=1,...,m$. If we know that $\operatorname{R}\in\Hom_\Fp(\Fp^m,\Fp^n)$, which means that each block entry satisfies $\operatorname{R}_{k,l}=[\operatorname{T}_{k,l},x_p]\in\Fp$, then by Lemma \ref{FpSelfNormLem} we get $\operatorname{T}_{k,l}\in\Fp$, meaning that $\operatorname{T}\in\Hom_\Fp(\Fp^m,\Fp^n)$. $\Box$
\end{proof}

For two $\F$-vector spaces $\mathbf{V}_1$, $\mathbf{V}_2$ and two linear operators $\operatorname{T}_1\in\End_\F(\mathbf{V}_1)$ and $\operatorname{T}_2\in\End_\F(\mathbf{V}_2)$, an \textit{intertwiner} of $\operatorname{T}_1$ and $\operatorname{T}_2$ is a linear operator $\Delta\in\Hom_\F(\mathbf{V}_1,\mathbf{V}_2)$ such that $\Delta\operatorname{T}_1=\operatorname{T}_2\Delta$. We denote by $\mathrm{C}_\F(\operatorname{T}_1,\operatorname{T}_2)$ the $\F$-vector space of all intertwiners between $\operatorname{T}_1$ and $\operatorname{T}_2$, and we set $\mathrm{C}_\F(\operatorname{T}_1)\doteq\mathrm{C}_\F(\operatorname{T}_1,\operatorname{T}_1)$. For a square matrix $\operatorname{T}\in\End_\F(\F^n)$ we will denote by $\mathrm{0}\llcorner\operatorname{T}\in\Hom_\F(\F^{n+l},\F^{n+m})$ the rectangular matrix with $\operatorname{T}$ occupying its top right corner and zeros elsewhere.

First we find all intertwiners between two Jordan blocks.
\begin{lemma}\label{NCommutantLem}For every $m,n\in\mathbb{N}$,
$$
\mathrm{C}_\F(\operatorname{N}_m,\operatorname{N}_n)=\left\{\mathrm{0}\llcorner\Delta\in\Hom_\F(\F^m,\F^n)\,\vline\quad\Delta\in\F[\operatorname{N}_{\min\{m,n\}}]\right\}.
$$
\end{lemma}
\begin{proof}This can be proven by direct computation. The commutation relation $\Delta\operatorname{N}_m=\operatorname{N}_n\Delta$ for a matrix $\Delta\in\Hom_\F(\F^m,\F^n)$ literarily means $\Delta_{k,l-1}=\Delta_{k+1,l}$ for all $1\le k\le n$, $1\le l\le m$ if we set for consistency $\Delta_{k,0}=\Delta_{n+1,l}=0$. This forces $\Delta$ to be constant on the diagonal and all superdiagonals of its top right largest square minor, and zero everywhere else, which is equivalent to the assertion. $\Box$
\end{proof}
\begin{lemma}\label{DeltaxpNestedCommLemma} For every $\Delta\in\mathrm{C}_\F(\operatorname{J}(p,m),\operatorname{J}(p,n))$ and $k\in\mathbb{N}$, the $k$-fold nested commutator satisfies
$$
\Delta_k\doteq[...[\Delta,x_p\id],...,x_p\id]=\sum_{l=0}^k(-1)^l\binom{k}{l}(\operatorname{N}_n\otimes\id_{\deg p})^{k-l}\Delta(\operatorname{N}_m\otimes\id_{\deg p})^l.
$$
\end{lemma}
\begin{proof} The proof is based on induction by $k$. From $\Delta\operatorname{J}(p,m)-\operatorname{J}(p,n)\Delta=0$ and 
$$
\operatorname{J}(p,n)=\id_n\otimes x_p+\operatorname{N}_n\otimes\id_{\deg p},\quad\forall n\in\mathbb{N},
$$
we verify the statement for $k=1$,
$$
\Delta_1=[\Delta,x_p\id]=\Delta(\id_m\otimes x_p)-(\id_n\otimes x_p)\Delta=(\operatorname{N}_n\otimes\id_{\deg p})\Delta-\Delta(\operatorname{N}_m\otimes\id_{\deg p}).
$$
Now suppose that the statement is true for some $k\in\mathbb{N}$. We establish that
$$
\Delta_{k+1}=\Delta_k(\id_m\otimes x_p)-(\id_n\otimes x_p)\Delta_k
$$
$$
=\sum_{l=0}^k(-1)^l\binom{k}{l}(\operatorname{N}_n\otimes\id_{\deg p})^{k-l}[\Delta,x_p\id](\operatorname{N}_m\otimes\id_{\deg p})^l
$$
$$
=\sum_{l=0}^k(-1)^l\binom{k}{l}(\operatorname{N}_n\otimes\id_{\deg p})^{k+1-l}\Delta(\operatorname{N}_m\otimes\id_{\deg p})^l
$$
$$
+\sum_{l=1}^{k+1}(-1)^l\binom{k}{l}(\operatorname{N}_n\otimes\id_{\deg p})^{k+1-l}\Delta(\operatorname{N}_m\otimes\id_{\deg p})^l
$$
$$
=\sum_{l=0}^{k+1}(-1)^l\binom{k+1}{l}(\operatorname{N}_n\otimes\id_{\deg p})^{k+1-l}\Delta(\operatorname{N}_m\otimes\id_{\deg p})^l,
$$
where in the last step we used the familiar properties of binomial coefficients,
$$
\binom{k}{l}+\binom{k}{l-1}=\binom{k+1}{l},\quad\binom{k}{k+1}=\binom{k}{-1}=0.
$$
The proof is complete. $\Box$
\end{proof}
\begin{proposition}\label{CJpnJqmProp} For every $p,q\in\sigma_\F$ and $m,n\in\mathbb{N}$,
$$
\mathrm{C}_\F(\operatorname{J}(q,m),\operatorname{J}(p,n))=\delta_{p,q}\mathrm{C}_\Fp(\operatorname{N}_m,\operatorname{N}_n).
$$
\end{proposition}
\begin{proof} Let $\Delta\operatorname{J}(q,m)=\operatorname{J}(p,n)\Delta$. Then $\Delta q^m(\operatorname{J}(q,m))=0=q^m(\operatorname{J}(p,n))\Delta$. But if $p\neq q$ then $q^m(\operatorname{J}(p,n))$ is invertible because it is upper block-triangular with all diagonal blocks being equal to $0\neq q^m(x_p)\in\Fp$, which forces $\Delta=0$. This explains the factor $\delta_{p,q}$ in the statement. Now assume $p=q$ and let $\Delta\operatorname{J}(p,m)=\operatorname{J}(p,n)\Delta$. Then by Lemma \ref{DeltaxpNestedCommLemma} we have
\begin{equation}
\Delta_k\doteq[...[\Delta,x_p\id],...,x_p\id]=\sum_{l=0}^k(-1)^l\binom{k}{l}(\operatorname{N}_n\otimes\id_{\deg p})^{k-l}\Delta(\operatorname{N}_m\otimes\id_{\deg p})^l.\label{NestedCommDeltaxp}
\end{equation}
If we set $k=m+n$ then the right hand side of (\ref{NestedCommDeltaxp}) vanishes, because each summand contains a factor $\mbox{N}_n^k$ with $k\ge n$. This gives $\Delta_{m+n}=[\Delta_{m+n-1},x_p\id]=0$. Now if we apply Corollary \ref{xpCommFptoFpCor} to $\Delta_k$ iteratively from $k=m+n-1$ to $k=1$ we will eventually establish that $\Delta\in\Hom_\Fp(\Fp^m,\Fp^n)$. This will also imply that
$$
\operatorname{N}_n\Delta-\Delta\operatorname{N}_m=[\Delta,x_p\id]=0\in\End_\Fp(\Fp^m,\Fp^n),
$$
so that $\Delta\in\mathrm{C}_\Fp(\operatorname{N}_m,\operatorname{N}_n)$, as desired. Conversely, every $\Delta\in\mathrm{C}_\Fp(\operatorname{N}_m,\operatorname{N}_n)$ satisfies $\Delta\mbox{J}(p,m)=\mbox{J}(p,n)\Delta$. $\Box$
\end{proof}

Next we will describe the explicit solutions of certain linear equations involving Jordanable operators. Let us first introduce the set of invertible matrices
$$
\operatorname{V}_n(\lambda)\doteq\diag(\lambda,\ldots,\lambda^n)=\begin{pmatrix}
\lambda & 0 & \ldots & 0\\
0 & \lambda^2 & \ldots & 0\\
\ldots & \ldots & \ldots & \ldots\\
0 & 0 & \ldots & \lambda^n
\end{pmatrix},\quad\forall\lambda\in\F^*,\quad\forall n\in\mathbb{N}.
$$
These matrices have a few useful properties.
\begin{lemma}\label{xlambdapLemma} For all $p\in\sigma_\F$ and $\lambda\in\F^*$ we have
$$
\operatorname{V}_{\deg p}(\lambda|\lambda|^{\epsilon-1})x_{\lambda\star p}=\lambda x_p\operatorname{V}_{\deg p}(\lambda|\lambda|^{\epsilon-1})\in\End_\F(\Fp),
$$
that is, $\operatorname{V}_{\deg p}(\lambda|\lambda|^{\epsilon-1})\in\mathrm{C}_\F(x_{\lambda\star p},\lambda x_p)$, where $\epsilon$ is per (\ref{epsilonDef}).
\end{lemma}
\begin{proof} Proof is by direct computation of matrix products. $\Box$
\end{proof}
\begin{lemma}\label{lambdaJLemma} For all $n\in\mathbb{N}$, $p\in\sigma_\F$ and $\lambda\in\F^*$ we have
$$
\lambda\operatorname{J}(p,n)\bigl(\operatorname{V}_n(\lambda^{-1})\otimes\operatorname{V}_{\deg p}(\lambda|\lambda|^{\epsilon-1})\bigr)=\bigl(\operatorname{V}_n(\lambda^{-1})\otimes\operatorname{V}_{\deg p}(\lambda|\lambda|^{\epsilon-1})\bigr)\operatorname{J}(\lambda\star p,n),
$$
that is, $\operatorname{V}_n(\lambda^{-1})\otimes\operatorname{V}_{\deg p}(\lambda|\lambda|^{\epsilon-1})\in\mathrm{C}_\F(\operatorname{J}(\lambda\star p,n),\lambda\operatorname{J}(p,n))$, where $\epsilon$ is per (\ref{epsilonDef}).
\end{lemma}
\begin{proof} Proof is by direct computation of products of block matrices,
$$
\lambda\operatorname{J}(p,n)\bigl(\operatorname{V}_n(\lambda^{-1})\otimes\operatorname{V}_{\deg p}(\lambda|\lambda|^{\epsilon-1})\bigr)
$$
$$
=\begin{pmatrix}
\lambda x_p & \lambda\id_{\deg p} & 0 & \ldots & 0\\
0 & \lambda x_p & \lambda\id_{\deg p} & \ldots & 0\\
\ldots & \ldots & \ldots & \ldots & \ldots\\
0 & 0 & 0 & \ldots & \lambda x_p
\end{pmatrix}\begin{pmatrix}
\lambda^{-1}\operatorname{V}_{\deg p}(\lambda|\lambda|^{\epsilon-1}) & 0 & 0 & \ldots & 0\\
0 & \lambda^{-2}\operatorname{V}_{\deg p}(\lambda|\lambda|^{\epsilon-1}) & 0 & \ldots & 0\\
\ldots & \ldots & \ldots & \ldots & \ldots\\
0 & 0 & 0 & \ldots & \lambda^{-n}\operatorname{V}_{\deg p}(\lambda|\lambda|^{\epsilon-1})
\end{pmatrix}
$$
$$
=\begin{pmatrix}
x_p\operatorname{V}_{\deg p}(\lambda|\lambda|^{\epsilon-1}) & \lambda^{-1}\operatorname{V}_{\deg p}(\lambda|\lambda|^{\epsilon-1}) & 0 & \ldots & 0\\
0 & \lambda^{-1}x_p\operatorname{V}_{\deg p}(\lambda|\lambda|^{\epsilon-1}) & \lambda^{-2}\operatorname{V}_{\deg p}(\lambda|\lambda|^{\epsilon-1}) & \ldots & 0\\
\ldots & \ldots & \ldots & \ldots & \ldots\\
0 & 0 & 0 & \ldots & \lambda^{-n+1}x_p\operatorname{V}_{\deg p}(\lambda|\lambda|^{\epsilon-1})
\end{pmatrix}
$$
$$
=\begin{pmatrix}
\lambda^{-1}\operatorname{V}_{\deg p}(\lambda|\lambda|^{\epsilon-1})x_{\lambda\star p} & \lambda^{-1}\operatorname{V}_{\deg p}(\lambda|\lambda|^{\epsilon-1}) & 0 & \ldots & 0\\
0 & \lambda^{-2}\operatorname{V}_{\deg p}(\lambda|\lambda|^{\epsilon-1})x_{\lambda\star p} & \lambda^{-2}\operatorname{V}_{\deg p}(\lambda|\lambda|^{\epsilon-1}) & \ldots & 0\\
\ldots & \ldots & \ldots & \ldots & \ldots\\
0 & 0 & 0 & \ldots & \lambda^{-n}\operatorname{V}_{\deg p}(\lambda|\lambda|^{\epsilon-1})x_{\lambda\star p}
\end{pmatrix}
$$
$$
=\begin{pmatrix}
\lambda^{-1}\operatorname{V}_{\deg p}(\lambda|\lambda|^{\epsilon-1}) & 0 & 0 & \ldots & 0\\
0 & \lambda^{-2}\operatorname{V}_{\deg p}(\lambda|\lambda|^{\epsilon-1}) & 0 & \ldots & 0\\
\ldots & \ldots & \ldots & \ldots & \ldots\\
0 & 0 & 0 & \ldots & \lambda^{-n}\operatorname{V}_{\deg p}(\lambda|\lambda|^{\epsilon-1})
\end{pmatrix}\begin{pmatrix}
x_{\lambda\star p} & \id_{\deg p} & 0 & \ldots & 0\\
0 & x_{\lambda\star p} & \id_{\deg p} & \ldots & 0\\
\ldots & \ldots & \ldots & \ldots & \ldots\\
0 & 0 & 0 & \ldots & x_{\lambda\star p}
\end{pmatrix}
$$
$$
=\bigl(\operatorname{V}_n(\lambda^{-1})\otimes\operatorname{V}_{\deg p}(\lambda|\lambda|^{\epsilon-1})\bigr)\operatorname{J}(\lambda\star p,n),
$$
where in the third step we used Lemma \ref{xlambdapLemma}. $\Box$
\end{proof}

We are now ready to solve the equation $\operatorname{X}\operatorname{T}=\lambda\operatorname{T}\operatorname{X}$ for a Jordanable operator $\operatorname{T}$ and a nonzero number $\lambda\in\F^*$.
\begin{proposition}\label{lambdaCommProp} For a Jordanable operator $\operatorname{T}=\operatorname{S}^{-1}\operatorname{J}(\aleph)\operatorname{S}$ and number $\lambda\in\F^*$ it holds
$$
\mathrm{C}_\F(\operatorname{T},\lambda\operatorname{T})=\operatorname{S}^{-1}\mathrm{C}_\F(\operatorname{J}(\aleph),\lambda\operatorname{J}(\aleph))\operatorname{S},
$$
$$
\mathrm{C}_\F(\operatorname{J}(\aleph),\lambda\operatorname{J}(\aleph))=\operatorname{V}(\lambda;\aleph)\cdot\mathrm{C}_\F(\operatorname{J}(\aleph),\operatorname{J}(\lambda\star\aleph)),
$$
$$
\operatorname{V}(\lambda;\aleph)\doteq\bigoplus_{p\in\sigma_\F}\bigoplus_{n=1}^\infty\bigoplus_{\aleph(p,n)}\bigl(\operatorname{V}_n(\lambda^{-1})\otimes\operatorname{V}_{\deg p}(\lambda|\lambda|^{\epsilon-1})\bigr),
$$
where $\epsilon$ is per (\ref{epsilonDef}). Every operator $\operatorname{R}\in\mathrm{C}_\F(\operatorname{J}(\aleph),\operatorname{J}(\lambda\star\aleph))$ is a block operator with arbitrary blocks of the form
$$
\operatorname{R}_{q,n,\beta;p,m,\alpha}\in\delta_{p,\lambda\star q}\mathrm{C}_\Fp(\operatorname{N}_m,\operatorname{N}_n),\quad\forall p,q\in\sigma_\F,\quad\forall m,n\in\mathbb{N},\quad\forall\alpha\in\aleph(p,m),\quad\forall\beta\in\aleph(q,n)
$$
such that there are finitely many nenzero entries in every column of $\operatorname{R}$.
\end{proposition}
\begin{proof} The original equation $\operatorname{X}\operatorname{T}=\lambda\operatorname{T}\operatorname{X}$ with the substitution $\operatorname{T}=\operatorname{S}^{-1}\operatorname{J}(\aleph)\operatorname{S}$ gives $\operatorname{X}=\operatorname{S}^{-1}\Delta\operatorname{S}$ where the operator $\Delta$ satisfies $\Delta\operatorname{J}(\aleph)=\lambda\operatorname{J}(\aleph)\Delta$. A blockwise application of Lemma \ref{lambdaJLemma} gives $\lambda\operatorname{J}(\aleph)\operatorname{V}(\lambda;\aleph)=\operatorname{V}(\lambda;\aleph)\operatorname{J}(\lambda\star\aleph)$ with $\operatorname{V}(\lambda;\aleph)$ as in the statement. This implies that $\Delta=\operatorname{V}(\lambda;\aleph)\operatorname{R}$ where the operator $\operatorname{R}$ satisfies $\operatorname{R}\operatorname{J}(\aleph)=\operatorname{J}(\lambda\star\aleph)\operatorname{R}$. If we break down the operator $\operatorname{R}$ into blocks $\operatorname{R}_{q,n,\beta;p,m,\alpha}$ then the latter equation reduces to
$$
\operatorname{R}_{q,n,\beta;p,m,\alpha}\operatorname{J}(p,m)=\operatorname{J}(\lambda\star q,n)\operatorname{R}_{q,n,\beta;p,m,\alpha},
$$
or
$$
\operatorname{R}_{q,n,\beta;p,m,\alpha}\in\mathrm{C}_\F(\operatorname{J}(p,m),\operatorname{J}(\lambda\star q,n))=\delta_{p,\lambda\star q}\mathrm{C}_\Fp(\operatorname{N}_m,\operatorname{N}_n),
$$
were the second step follows from Proposition \ref{CJpnJqmProp}. $\Box$
\end{proof}
Next we introduce for convenience the following sequence of matrices,
$$
\operatorname{U}_n\doteq\frac{d}{d\lambda}\left[\lambda^n\operatorname{V}_n(\lambda^{-1})\right](1)=\begin{pmatrix}
n-1 & 0 & \ldots & 0\\
0 & n-2 & \ldots & 0\\
\ldots & \ldots & \ldots & \ldots\\
0 & 0 & 0 & 0
\end{pmatrix},\quad\forall n\in\mathbb{N}.
$$
These matrices are good for the following property.
\begin{lemma}\label{U_nN_nCommLemma} For all $n\in\mathbb{N}$ we have
$$
\operatorname{U}_n\operatorname{N}_n-\operatorname{N}_n\operatorname{U}_n=\operatorname{N}_n.
$$
\end{lemma}
\begin{proof}
The proof is by direct computation of matrix products. $\Box$
\end{proof}

Let us proceed to the equation $\operatorname{Y}\operatorname{T}-\operatorname{T}\operatorname{Y}=\operatorname{T}$ for a Jordanable operator $\operatorname{T}$. The monic irreducible polynomial $p(X)=X$ will be denoted simply by $X\in\sigma_\F$.
\begin{proposition}\label{InhomCommProp} For a nonzero Jordanable operator $\operatorname{T}=\operatorname{S}^{-1}\operatorname{J}(\aleph)\operatorname{S}$ the equation $\operatorname{Y}\operatorname{T}-\operatorname{T}\operatorname{Y}=\operatorname{T}$ has solutions $\operatorname{Y}$ if and only if $\supp\aleph=\{X\}$, in which case the set of all solutions is
$$
\operatorname{S}^{-1}\bigl(\operatorname{U}(\aleph)+\mathrm{C}_\F(\operatorname{J}(\aleph))\bigr)\operatorname{S},
$$
$$
\operatorname{U}(\aleph)\doteq\bigoplus_{n=1}^\infty\bigoplus_{\aleph(X,n)}\operatorname{U}_n
$$
where $\mathrm{C}_\F(\operatorname{J}(\aleph))=\mathrm{C}_\F(\operatorname{J}(\aleph),\operatorname{J}(\aleph))$ can be found using Proposition \ref{lambdaCommProp} with $\lambda=1$.
\end{proposition}
\begin{proof} The original equation $\operatorname{Y}\operatorname{T}-\operatorname{T}\operatorname{Y}=\operatorname{T}$ with the substitution $\operatorname{T}=\operatorname{S}^{-1}\operatorname{J}(\aleph)\operatorname{S}$ gives $\operatorname{Y}=\operatorname{S}^{-1}\Delta\operatorname{S}$ where the operator $\Delta$ satisfies $\Delta\operatorname{J}(\aleph)-\operatorname{J}(\aleph)\Delta=\operatorname{J}(\aleph)$. If we decompose the operator $\Delta$ into blocks $\Delta_{q,n,\beta;p,m,\alpha}$ then the latter equation reduces to
$$
\Delta_{q,n,\beta;p,m,\alpha}\operatorname{J}(p,m)-\operatorname{J}(q,n)\Delta_{q,n,\beta;p,m,\alpha}=\delta_{q,p}\delta_{n,m}\delta_{\beta,\alpha}\operatorname{J}(p,m),
$$
$$
\forall p,q\in\sigma_\F,\quad\forall m,n\in\mathbb{N},\quad\forall\alpha\in\aleph(p,m),\quad\forall\beta\in\aleph(q,n).
$$
On the diagonal $q=p$, $n=m$, $\beta=\alpha$ we get
\begin{equation}
\Delta_{p,m,\alpha;p,m,\alpha}\operatorname{J}(p,m)-\operatorname{J}(p,m)\Delta_{p,m,\alpha;p,m,\alpha}=\operatorname{J}(p,m)\in\Hom_\F(\Fp^m,\Fp^m),\label{DJCommJ}
\end{equation}
which is equivalent to
\begin{equation}
[\Delta_{p,m,\alpha;p,m,\alpha},x_p\id]=[(\operatorname{N}_m\otimes\id_{\deg p}),\Delta_{p,m,\alpha;p,m,\alpha}]+\operatorname{J}(p,m).\label{DxpNDJ}
\end{equation}
From (\ref{DJCommJ}) one can show inductively that
$$
\Delta_{p,m,\alpha;p,m,\alpha}\operatorname{J}(p,m)^k-\operatorname{J}(p,m)^k\Delta_{p,m,\alpha;p,m,\alpha}=k\operatorname{J}(p,m)^k,\quad\forall k\in\mathbb{N},
$$
and thus
$$
\Delta_{p,m,\alpha;p,m,\alpha}P(\operatorname{J}(p,m))-P(\operatorname{J}(p,m))\Delta_{p,m,\alpha;p,m,\alpha}=P'(\operatorname{J}(p,m))\operatorname{J}(p,m),\quad\forall P(X)\in\F[X].
$$
Choosing $P(X)=p(X)$ we obtain
$$
\Delta_{p,m,\alpha;p,m,\alpha}(\operatorname{N}_m\otimes\id_{\deg p})-(\operatorname{N}_m\otimes\id_{\deg p})\Delta_{p,m,\alpha;p,m,\alpha}=p'(\operatorname{J}(p,m))\operatorname{J}(p,m),
$$
which combined with (\ref{DxpNDJ}) yields
$$
[\Delta_{p,m,\alpha;p,m,\alpha},x_p\id]=-p'(\operatorname{J}(p,m))\operatorname{J}(p,m)+\operatorname{J}(p,m)\in\Hom_\Fp(\Fp^m,\Fp^m).
$$
Therefore by Corollary \ref{xpCommFptoFpCor} we find that $\Delta_{p,m,\alpha;p,m,\alpha}\in\Hom_\Fp(\Fp^m,\Fp^m)$, i.e., it is $\Fp$-linear. Thus we can view (\ref{DJCommJ}) as an $\Fp$-linear equation and take the $\Fp$-trace of both parts,
$$
\tr_\Fp[\Delta_{p,m,\alpha;p,m,\alpha},\operatorname{J}(p,m)]=0=\tr_\Fp\operatorname{J}(p,m)=mx_p,
$$
whence $x_p=0$ and $p(X)=X$. This proves that if a solution $\operatorname{Y}$ exists then necessarily $\supp\aleph=\{X\}$.

Now assume that $\supp\aleph=\{X\}$. The equation $[\Delta,\operatorname{J}(\aleph)]=\operatorname{J}(\aleph)$ is linear inhomogeneous, and its general solution is of the form
$$
\Delta=\Delta_0+\operatorname{X},
$$
where $\Delta_0$ is a particular solution and $\operatorname{X}$ is an arbitrary solution of the homogeneous equation $[\operatorname{X},\operatorname{J}(\aleph)]=0$, i.e.,
$$
\Delta\in\Delta_0+\mathrm{C}_\F(\operatorname{J}(\aleph)),
$$
as desired. It remains to show that
$$
\Delta_0=\operatorname{U}(\aleph)=\bigoplus_{n=1}^\infty\bigoplus_{\aleph(X,n)}\operatorname{U}_n
$$
is indeed a particular solution,
$$
[\operatorname{U}(\aleph),\operatorname{J}(\aleph)]=\bigoplus_{n=1}^\infty\bigoplus_{\aleph(X,n)}[\operatorname{U}_n,\operatorname{N}_n]=\bigoplus_{n=1}^\infty\bigoplus_{\aleph(X,n)}\operatorname{N}_n=\operatorname{J}(\aleph),
$$
where $\operatorname{J}(X,n)=\operatorname{N}_n$ and Lemma \ref{U_nN_nCommLemma} were used. $\Box$
\end{proof}

Our final task in this section will be to solve the equation $\operatorname{Z}\operatorname{T}+\operatorname{T}^\top\operatorname{Z}=0$ for a Jordanable operator $\operatorname{T}$. Since the transpose of a linear operator does not generally make an invariant sence in infinite dimensions (but only in a basis where the operator has finitely many nonzero entries in every row), we restrict to the case $\operatorname{T}=\operatorname{J}(\aleph)$. To this avail we first introduce the matrices
$$
\operatorname{P}_n=\begin{pmatrix}
0 & \ldots & 0 & 1\\
0 & \ldots & 1 & 0\\
\ldots & \ldots & \ldots & \ldots\\
1 & \ldots & 0 & 0
\end{pmatrix}\in\End_\F(\F^n),\quad\forall n\in\mathbb{N}.
$$
For $n=1$ we have $\operatorname{P}_1=1$.
\begin{lemma}\label{PnNnLemma} For all $n\in\mathbb{N}$ we have
$$
\operatorname{P}_n\operatorname{N}_n=\operatorname{N}_n^\top\operatorname{P}_n.
$$
\end{lemma}
\begin{proof} The proof is by direct computation of matrix products. $\Box$
\end{proof}

For every $p\in\sigma_\F$ let
$$
p(X)=X^d+a_{d-1}X^{d-1}+\ldots+a_0,\quad d=\deg p,
$$
and define $d-1$ numbers $\mu_1,\ldots,\mu_{d-1}\in\F$ recursively by
\begin{equation}
\mu_n=-a_{d-n}-\sum_{k=1}^{n-1}a_{d-n+k}\,\mu_k,\quad n=1,\ldots,d-1.\label{mundef}
\end{equation}
We introduce the symmetric matrix
$$
\operatorname{W}^\epsilon_p=\begin{pmatrix}
0 & 0 & \ldots & 0 & 1\\
0 & 0 & \ldots & 1 & \epsilon\mu_1\\
0 & 0 & \ldots & \epsilon\mu_1 & \mu_2\\
\ldots & \ldots & \ldots & \ldots & \ldots\\
1 & \epsilon\mu_1 & \ldots & \mu_{d-2} & \mu_{d-1}
\end{pmatrix}\in\End_\F(\Fp).
$$
If $\deg p=1$ then $x_p\in\F$ and $\operatorname{W}^\epsilon_p=1$.
\begin{lemma}\label{WpxpLemma} For all $p\in\sigma_f$ we have
$$
\operatorname{W}^\epsilon_px_p=x_p^\top\operatorname{W}^\epsilon_p\in\End_\F(\Fp),
$$
where $\epsilon$ is per (\ref{epsilonDef}).
\end{lemma}
\begin{proof}The proof is by direct computation of matrix products. If $\epsilon=1$ then
$$
\operatorname{W}^1_px_p=\begin{pmatrix}
0 & 0 & \ldots & 0 & 0 & 1\\
0 & 0 & \ldots & 0 & 1 & \mu_1\\
0 & 0 & \ldots & 1 & \mu_1 & \mu_2\\
\ldots & \ldots & \ldots & \ldots & \ldots & \ldots\\
1 & \mu_1 & \ldots & \mu_{d-3} & \mu_{d-2} & \mu_{d-1}
\end{pmatrix}\begin{pmatrix}
0 & 0 & \ldots & 0 & 0 & -a_0\\
1 & 0 & \ldots & 0 & 0 & -a_1\\
0 & 1 & \ldots & 0 & 0 & -a_2\\
\ldots & \ldots & \ldots & \ldots & \ldots & \ldots\\
0 & 0 & \ldots & 0 & 1 & -a_{d-1}
\end{pmatrix}
$$
$$
=\begin{pmatrix}
0 & 0 & \ldots & 0 & 1 & -a_{d-1}\\
0 & 0 & \ldots & 1 & \mu_1 & -a_{d-2}-\mu_1a_{d-1}\\
0 & 0 & \ldots & \mu_1 & \mu_2 & -a_{d-3}-\mu_1a_{d-2}-\mu_2a_{d-1}\\
\ldots & \ldots & \ldots & \ldots & \ldots & \ldots\\
\mu_1 & \mu_2 & \ldots & \mu_{d-2} & \mu_{d-1} & -a_0-\mu_1a_1-\ldots-\mu_{d-1}a_{d-1}
\end{pmatrix}
$$
$$
=\begin{pmatrix}
0 & 0 & \ldots & 0 & 1 & \mu_1\\
0 & 0 & \ldots & 1 & \mu_1 & \mu_2\\
0 & 0 & \ldots & \mu_1 & \mu_2 & \mu_3\\
\ldots & \ldots & \ldots & \ldots & \ldots & \ldots\\
\mu_1 & \mu_2 & \ldots & \mu_{d-2} & \mu_{d-1} & -a_0-\mu_1a_1-\ldots-\mu_{d-1}a_{d-1}
\end{pmatrix}
$$
$$
=\begin{pmatrix}
0 & 1 & 0 & \ldots & 0 & 0\\
0 & 0 & 1 & \ldots & 0 & 0\\
0 & 0 & 0 & \ldots & 0 & 0\\
\ldots & \ldots & \ldots & \ldots & \ldots & \ldots\\
-a_0 & -a_1 & -a_2 & \ldots & -a_{d-2} & -a_{d-1}
\end{pmatrix}\begin{pmatrix}
0 & 0 & \ldots & 0 & 0 & 1\\
0 & 0 & \ldots & 0 & 1 & \mu_1\\
0 & 0 & \ldots & 1 & \mu_1 & \mu_2\\
\ldots & \ldots & \ldots & \ldots & \ldots & \ldots\\
1 & \mu_1 & \ldots & \mu_{d-3} & \mu_{d-2} & \mu_{d-1}
\end{pmatrix}=x_p^\top\operatorname{W}^1_p,
$$
where the defintion (\ref{mundef}) of the numbers $\mu_n$ was used in the third equality. If on the other hand $\epsilon=0$ then
$$
\operatorname{W}^0_px_p=\begin{pmatrix}
0 & 1\\
1 & 0
\end{pmatrix}\begin{pmatrix}
a & -b\\
b & a
\end{pmatrix}=\begin{pmatrix}
b & a\\
a & -b
\end{pmatrix}=\begin{pmatrix}
a & b\\
-b & a
\end{pmatrix}\begin{pmatrix}
0 & 1\\
1 & 0
\end{pmatrix}=x_p^\top\operatorname{W}^0_p,
$$ 
as desired. $\Box$
\end{proof}
\begin{proposition}\label{TTtransCommProp} For a nonzero multiplicity function $\aleph$, the solutions of the\newline equation $\operatorname{Z}\operatorname{J}(\aleph)+\operatorname{J}(\aleph)^\top\operatorname{Z}=0$ are
$$
\operatorname{Z}=\operatorname{W}(\aleph)\operatorname{A},
$$
where
$$
\operatorname{W}(\aleph)=\bigoplus_{p\in\sigma_\F}\bigoplus_{n=1}^\infty\bigoplus_{\aleph(p,n)}\operatorname{P}_n\otimes\operatorname{W}^\epsilon_p,
$$
$\epsilon$ is per (\ref{epsilonDef}) and $\operatorname{A}\in\mathrm{C}_\F(\operatorname{J}(\aleph),-\operatorname{J}(\aleph))$ as in Proposition \ref{lambdaCommProp} with $\lambda=-1$.
\end{proposition}
\begin{proof} Using Lemma \ref{PnNnLemma} and Lemma \ref{WpxpLemma} we establish for every $n\in\mathbb{N}$ and $p\in\sigma_\F$ that
$$
(\operatorname{P}_n\otimes\operatorname{W}^\epsilon_p)\operatorname{J}(p,n)=(\operatorname{P}_n\otimes\operatorname{W}^\epsilon_p)(\id_n\otimes x_p+\operatorname{N}_n\otimes\id_{\deg p})=\operatorname{P}_n\otimes(\operatorname{W}^\epsilon_px_p)+(\operatorname{P}_n\operatorname{N}_n)\otimes\operatorname{W}^\epsilon_p
$$
$$
=\operatorname{P}_n\otimes(x_p^\top\operatorname{W}^\epsilon_p)+(\operatorname{N}_n^\top\operatorname{P}_n)\otimes\operatorname{W}^\epsilon_p=(\id_n\otimes x_p^\top+\operatorname{N}_n^\top\otimes\id_{\deg p})(\operatorname{P}_n\otimes\operatorname{W}^\epsilon_p)=\operatorname{J}(p,n)^\top(\operatorname{P}_n\otimes\operatorname{W}^\epsilon_p),
$$
which implies $\operatorname{W}(\aleph)\operatorname{J}(\aleph)=\operatorname{J}(\aleph)^\top\operatorname{W}(\aleph)$ with operator $\operatorname{W}(\aleph)$ as in the statement. The operator $\operatorname{W}(\aleph)$ is invertible, and we are free to make a substitution $\operatorname{Z}=\operatorname{W}(\aleph)\operatorname{A}$, which yields $\operatorname{A}\operatorname{J}(\aleph)+\operatorname{J}(\aleph)\operatorname{A}=0$, i.e., $\operatorname{A}\in\mathrm{C}_\F(\operatorname{J}(\aleph),-\operatorname{J}(\aleph))$ as desired. $\Box$
\end{proof}





\section{Jordanable almost Abelian Lie algebras}\label{JordAALieAlg}

We finally come to the main subject of this paper, which is the study of Jordanable almost Abelian Lie algebras. An almost Abelian Lie algebra is a non-Abelian Lie algebra over a field $\F$ which contains a codimension 1 Abelian subalgebra. We refer the reader to \cite{Avetisyan2016} for all relevant definitions and facts about this class of Lie algebras that are going to be used here. An almost Abelian Lie algebra can be written as the semidirect product $\mathbf{L}=\F e_0\rtimes\mathbf{V}$ of the 1-dimensional Abelian Lie algebra $\F e_0$ with a $(\dim_\F\mathbf{L}-1)$-dimensional Abelian Lie algebra $\mathbf{V}$. The Lie algebra structure is completely determined by the nonzero operator $\ad_{e_0}\in\End_\F(\mathbf{V})$ defined by
$$
\ad_{e_0}v=[e_0,v],\quad\forall v\in\mathbf{V}.
$$
In Proposition 11 of \cite{Avetisyan2016} it was shown that isomorphism classes of almost Abelian Lie algebras are in a bijective correspondence with projective similarity classes $[\F^*\ad_{e_0}]$. Here we want to consider a subclass of almost Abelian Lie algebras which allows for a more explicit description of structure than in \cite{Avetisyan2016}. In particular, several results in that paper are formulated in terms of solutions to operator (matrix) equations, which are hard to solve explicitly in general. In order to achieve explicit formulae we assume that the operators (matrices) in question are Jordanable, which makes these matrix equations tractable algebraically.
\begin{definition} We will say that an almost Abelian Lie algebra $\mathbf{L}=\F e_0\rtimes\mathbf{V}$ is Jordanable if the operator $\ad_{e_0}$ over $\mathbf{V}$ is Jordanable.
\end{definition}
\begin{remark} In particular, every finite dimensional almost Abelian Lie algebra is Jordanable.
\end{remark}
If $\mathbf{L}=\F e_0\rtimes\mathbf{V}$ is a Jordanable almost Abelian Lie algebra then by Proposition 11 of \cite{Avetisyan2016} and Corollary \ref{ProjMultProjSimCor} of the last section we see that the composite map
\begin{equation}
[\mathbf{L}]\mapsto[\F^*\ad_{e_0}]=[\F^*\mbox{\normalfont J}](\aleph_{\ad_{e_0}})\mapsto\F^*\star\aleph_{\ad_{e_0}}\label{LtoProjMult}
\end{equation}
from isomorphism classes of Jordanable almost Abelian Lie algebras to projective multiplicity functions is a bijection. Denote by
$$
\aA_\F(\aleph)=\aA(\aleph)\doteq\F e_0\rtimes\mathbf{V},\quad\ad_{e_0}=\operatorname{J}(\aleph),\quad\mathbf{V}=\F^{\dim_\F\aleph}
$$
the Jordanable almost Abelian Lie algebra associated to the multiplicity function $\aleph$, and by $\aA(\F^*\star\aleph)=[\aA(\aleph)]$ its isomorphism class. Then the map $\F^*\star\aleph\mapsto\aA(\F^*\star\aleph)$ is the inverse of the map in (\ref{LtoProjMult}). In the sequel we will be interested only in almost Abelian Lie algebras up to isomorphism, and from every isomorphism class $\aA(\F^*\star\aleph)$ we will always choose the convenient representative $\aA(\aleph)$ defined above. Every Jordanable almost Abelian Lie algebra $\mathbf{L}=\F e_0\rtimes\mathbf{V}$ can be brought to this canonical form by choosing a suitable basis in $\mathbf{V}$ and rescaling $e_0$.

Let us now fix a non-zero multiplicity function $\aleph$ (remember that almost Abelian Lie algebras are assumed to be non-Abelian) and consider the associated canonical almost Abelian Lie algebra $\mathbf{L}=\aA(\aleph)=\F e_0\rtimes\mathbf{V}$. Let us for every $p\in\sigma_\F$, $n\in\mathbb{N}$ and $\alpha\in\aleph(p,n)$ choose standard $\Fp$-basis $\{e_\alpha^m(p,n)\}$, $m=1,...,n$ of the block space $\Fp^n$ so that
\begin{equation}
\mathbf{V}=\bigoplus_{p\in\sigma_\F}\bigoplus_{n=1}^\infty\bigoplus_{\alpha\in\aleph(p,n)}\bigoplus_{m=1}^n\Fp e_\alpha^m(p,n),\quad\operatorname{N}_ne_\alpha^{m}(p,n)=\begin{cases}
e_\alpha^{m-1}(p,n),\quad m>1,\\
0\quad m=1
\end{cases}.\label{AdaptBasis}
\end{equation}
According to Remark 2 from \cite{Avetisyan2016} we have for the centre $\mathcal{Z}(\mathbf{L})=\ker\ad_{e_0}$ and for the lower central ceries $\mathbf{L}_{(k)}=[\mathbf{L},...,[\mathbf{L},\mathbf{L}]...]=\ad_{e_0}^k\mathbf{V}$ for all $k\in\mathbb{N}$. Now $\ad_{e_0}=\operatorname{J}(\aleph)$ allows us to explicitly write
$$
\mathcal{Z}(\mathbf{L})=\bigoplus_{p\in\sigma_\F}\bigoplus_{n=1}^\infty\bigoplus_{\aleph(p,n)}\ker\operatorname{J}(p,n)=\bigoplus_{n=1}^\infty\bigoplus_{\aleph(X,n)}\ker\operatorname{N}_n,
$$
where we write $X$ for the polynomial $p(X)=X$ with $x_p=0$, and $\operatorname{J}(X,n)=\operatorname{N}_n$ by formula (\ref{JpnDef}). In the standard basis this reads
\begin{equation}
\mathcal{Z}(\mathbf{L})=\bigoplus_{n=1}^\infty\bigoplus_{\alpha\in\aleph(X,n)}\F e_\alpha^1(X,n)\label{ZLBasis}
\end{equation}
(we used the fact that $\Fp=\F$ for $p(X)=X$). Meanwhile
$$
\mathbf{L}_{(k)}=\bigoplus_{p\in\sigma_\F}\bigoplus_{n=1}^\infty\bigoplus_{\aleph(p,n)}\operatorname{J}(p,n)^k\Fp^n=\left[\bigoplus_{p\in\sigma_\F\setminus\{X\}}\bigoplus_{n=1}^\infty\bigoplus_{\aleph(p,n)}\Fp^n\right]\bigoplus\left[\bigoplus_{n=1}^\infty\bigoplus_{\aleph(X,n)}\operatorname{N}_n^k\F^n\right],
$$
which in the standard basis becomes
\begin{equation}
\mathbf{L}_{(k)}=\left[\bigoplus_{p\in\sigma_\F\setminus\{X\}}\bigoplus_{n=1}^\infty\bigoplus_{\aleph(p,n)}\bigoplus_{m=1}^n\Fp e_\alpha^m(p,n)\right]\bigoplus\left[\bigoplus_{n=1}^\infty\bigoplus_{\alpha\in\aleph(X,n)}\bigoplus_{m=1}^{n-k}\F e_\alpha^m(X,n)\right].\label{LkBasis}
\end{equation}
In particular, we arrive at the following.
\begin{remark}\label{NilpoRem} $\aA(\aleph)$ is nilpotent if and only if $\supp\aleph=\{X\}$.
\end{remark}
Let us turn now to the decomposition $\mathbf{L}=\mathbf{L}_0\oplus\mathbf{W}$ from Proposition 5 of \cite{Avetisyan2016}. Here only blocks with $p(X)=X$ and $n=1$ enter the direct summand $\mathbf{W}$,
\begin{equation}
\mathbf{W}=\bigoplus_{\aleph(X,1)}\F=\bigoplus_{\alpha\in\aleph(X,1)}\F e_\alpha^1(X,1),\label{LDecompBasis}
\end{equation}
and the rest falls into $\mathbf{L}_0$.

Another important structural property of a Lie algebra is the family of its Lie subalgebras and ideals. Proposition 4 in \cite{Avetisyan2016} describes these for a given almost Abelian Lie algebra in terms of the kernel and the image of $\ad_{e_0}$ and subspaces of $\mathbf{V}$ invariant under $\ad_{e_0}$. For a Jordanable almost Abelian Lie algebra an explicit description is available for all three constituents. The kernel and the image of $\ad_{e_0}$ were described above, and $\ad_{e_0}$-invariant subspaces were dealt with in Corollary \ref{InvSubspCorr}. One important fact implied by Proposition \ref{JordOpRestrProp} is the following.
\begin{remark} Every almost Abelian Lie subalgebra of a Jordanable almost Abelian Lie algebra is Jordanable.
\end{remark}

Next we will describe the automorphism group of a Jordanable almost Abelian Lie algebra other than the Heisenberg algebra with all details spelled out (for the Heisenberg algebra the answer is well known and can be found, for instance, in \cite{Avetisyan2016}). The following is an easy corollary of Lemma \ref{NCommutantLem}, Proposition \ref{lambdaCommProp} and Proposition 10 in \cite{Avetisyan2016}.
\begin{proposition}\label{AutaA} The automorphism group of the indecomposable Jordanable almost Abelian Lie algebra $\aA(\aleph)=\F e_0\rtimes\mathbf{V}$ with $\ad_{e_0}=\operatorname{J}(\aleph)$ other than $\mathbf{H}_\F$ is
$$
\Aut\left(\aA(\aleph)\right)=\left\{\begin{pmatrix}
\nu & 0\\
\gamma & \Delta
\end{pmatrix}\,\vline\quad\nu\in\Dil(\aleph),\quad\gamma\in\mathbf{V},\quad\Delta\in\Aut(\F^{\dim_\F\aleph})\right\},
$$
$$
\Delta_{q,n,\beta;p,m,\alpha}=\delta_{p,\nu\star q}\left(\operatorname{V}_n(\nu^{-1})\otimes\operatorname{V}_d(\nu|\nu|^{\epsilon-1})\right)\operatorname{R}_{p;n,\beta;m,\alpha}\in\Hom_\F(\Fp^m,\Fp^n),
$$

$$
\left(\operatorname{V}_n(\nu^{-1})\otimes\operatorname{V}_d(\nu|\nu|^{\epsilon-1})\right)=
$$

$$
\begin{pmatrix}
\begin{matrix}
1 & 0 & \ldots & 0\\
0 & \nu & \ldots & 0\\
\ldots & \ldots & \ldots & \ldots\\
0 & 0 & \ldots & \nu^{d-1}
\end{matrix} & 0 & \ldots & 0\\
0 & \begin{matrix}
\nu^{-1} & 0 & \ldots & 0\\
0 & 1 & \ldots & 0\\
\ldots & \ldots & \ldots & \ldots\\
0 & 0 & \ldots & \nu^{d-2}
\end{matrix} & \ldots & 0\\
\ldots & \ldots & \ldots & \ldots\\
0 & 0 & \ldots & \begin{matrix}
\nu^{1-n} & 0 & \ldots & 0\\
0 & \nu^{2-n} & \ldots & 0\\
\ldots & \ldots & \ldots & \ldots\\
0 & 0 & \ldots & \nu^{d-n}
\end{matrix}
\end{pmatrix}
$$

for $\epsilon=1$ and 

$$
\left(\operatorname{V}_n(\nu^{-1})\otimes\operatorname{V}_d(\nu|\nu|^{\epsilon-1})\right)=\begin{pmatrix}
\begin{matrix}
|\nu|^{-1} & 0\\
0 & \nu|\nu|^{-2}
\end{matrix} & 0 & \ldots & 0\\
0 & \begin{matrix}
\nu^{-1}|\nu|^{-1} & 0\\
0 & |\nu|^{-2}
\end{matrix} & \ldots & 0\\
\ldots & \ldots & \ldots & \ldots\\
0 & 0 & \ldots & \begin{matrix}
\nu^{1-n}|\nu|^{-1} & 0\\
0 & \nu^{2-n}|\nu|^{-2}\end{matrix}
\end{pmatrix}
$$

for $\epsilon=0$,

$$
\operatorname{R}_{p;n,\beta;m,\alpha}=\begin{pmatrix}
0 & \ldots & 0 & r_1 & r_2 & \ldots & r_l\\
0 & \ldots & 0 & 0 & r_1 & \ldots & r_{l-1}\\
\ldots & \ldots & \ldots & \ldots & \ldots & \ldots & \ldots\\
0 & \ldots & 0 & 0 & 0 & \ldots & r_1
\end{pmatrix}\in\End_\Fp(\Fp^m,\Fp^n)\quad\mbox{if}\quad m\ge n,
$$

$$
\operatorname{R}_{p;n,\beta;m,\alpha}=\begin{pmatrix}
r_1 & r_2 & \ldots & r_l\\
0 & r_1 & \ldots & r_{l-1}\\
\ldots & \ldots & \ldots & \ldots\\
0 & 0 & \ldots & r_1\\
0 & 0 & \ldots & 0\\
\ldots & \ldots & \ldots & \ldots\\
0 & 0 & \ldots & 0
\end{pmatrix}\in\End_\Fp(\Fp^m,\Fp^n)\quad\mbox{if}\quad m<n,
$$

$$
l=\min\{m,n\},\quad d=\deg p,\quad r_1,\ldots,r_l\in\Fp,
$$
and $\epsilon$ is per (\ref{epsilonDef}). For every $p,m,\alpha$ only finitely many $q,n,\beta$ correspond to nonzero blocks.
\end{proposition}

We can also describe the Lie algebra of derivations in a similar manner. Here we will distinguish between nilpotent and non-nilpotent Lie algebras $\aA(\aleph)$, see Remark \ref{NilpoRem}. The below results are an immediate consequence of Lemma \ref{NCommutantLem}, Proposition \ref{lambdaCommProp}, Proposition \ref{InhomCommProp} and Proposition 14 in \cite{Avetisyan2016}.
\begin{proposition}\label{DeraANonnil} The algebra of derivations of the indecomposable Jordanable non-nilpotent almost Abelian Lie algebra $\aA(\aleph)=\F e_0\rtimes\mathbf{V}$ with $\ad_{e_0}=\operatorname{J}(\aleph)$ is
$$
\Der\left(\aA(\aleph)\right)=\left\{\begin{pmatrix}
0 & 0\\
\gamma & \Delta
\end{pmatrix}\,\vline\quad\gamma\in\mathbf{V},\quad\Delta\in\End_\F(\F^{\dim_\F\aleph})\right\},
$$
$$
\Delta_{q,n,\beta;p,m,\alpha}=\delta_{p,q}\operatorname{R}_{p;n,\beta;m,\alpha}\in\Hom_\F(\Fp^m,\Fp^n),
$$

$$
\operatorname{R}_{p;n,\beta;m,\alpha}=\begin{pmatrix}
0 & \ldots & 0 & r_1 & r_2 & \ldots & r_l\\
0 & \ldots & 0 & 0 & r_1 & \ldots & r_{l-1}\\
\ldots & \ldots & \ldots & \ldots & \ldots & \ldots & \ldots\\
0 & \ldots & 0 & 0 & 0 & \ldots & r_1
\end{pmatrix}\in\End_\Fp(\Fp^m,\Fp^n)\quad\mbox{if}\quad m\ge n,
$$

$$
\operatorname{R}_{p;n,\beta;m,\alpha}=\begin{pmatrix}
r_1 & r_2 & \ldots & r_l\\
0 & r_1 & \ldots & r_{l-1}\\
\ldots & \ldots & \ldots & \ldots\\
0 & 0 & \ldots & r_1\\
0 & 0 & \ldots & 0\\
\ldots & \ldots & \ldots & \ldots\\
0 & 0 & \ldots & 0
\end{pmatrix}\in\End_\Fp(\Fp^m,\Fp^n)\quad\mbox{if}\quad m<n,
$$

$$
l=\min\{m,n\},\quad d=\deg p,\quad r_1,\ldots,r_l\in\Fp.
$$
For every $p,m,\alpha$ only finitely many $q,n,\beta$ correspond to nonzero blocks.
\end{proposition}
\begin{proposition} The algebra of derivations of the indecomposable Jordanable nilpotent almost Abelian Lie algebra $\aA(\aleph)=\F e_0\rtimes\mathbf{V}$ with $\ad_{e_0}=\operatorname{J}(\aleph)$ other than $\mathbf{H}_\F$ is
$$
\Der\left(\aA(\aleph)\right)=\left\{\begin{pmatrix}
\alpha & 0\\
\gamma & \Delta
\end{pmatrix}\,\vline\quad\alpha\in\F,\quad\gamma\in\mathbf{V},\quad\Delta\in\End_\F(\F^{\dim_\F\aleph})\right\},
$$
$$
\Delta_{X,n,\beta;X,m,\alpha}=\alpha\delta_{n,m}\operatorname{U}_n+\operatorname{R}_{n,\beta;m,\alpha}\in\Hom_\F(\F^m,\F^n),
$$

$$
\operatorname{U}_n=\begin{pmatrix}
n-1 & 0 & \ldots & 0\\
0 & n-2 & \ldots & 0\\
\ldots & \ldots & \ldots & \ldots\\
0 & 0 & \ldots & 0
\end{pmatrix}\in\End_\F(\F^n),
$$

$$
\operatorname{R}_{n,\beta;m,\alpha}=\begin{pmatrix}
0 & \ldots & 0 & r_1 & r_2 & \ldots & r_l\\
0 & \ldots & 0 & 0 & r_1 & \ldots & r_{l-1}\\
\ldots & \ldots & \ldots & \ldots & \ldots & \ldots & \ldots\\
0 & \ldots & 0 & 0 & 0 & \ldots & r_1
\end{pmatrix}\in\End_\F(\F^m,\F^n)\quad\mbox{if}\quad m\ge n,
$$

$$
\operatorname{R}_{n,\beta;m,\alpha}=\begin{pmatrix}
r_1 & r_2 & \ldots & r_l\\
0 & r_1 & \ldots & r_{l-1}\\
\ldots & \ldots & \ldots & \ldots\\
0 & 0 & \ldots & r_1\\
0 & 0 & \ldots & 0\\
\ldots & \ldots & \ldots & \ldots\\
0 & 0 & \ldots & 0
\end{pmatrix}\in\End_\F(\F^m,\F^n)\quad\mbox{if}\quad m<n,
$$

$$
l=\min\{m,n\},\quad r_1,\ldots,r_l\in\F.
$$
For every $m,\alpha$ only finitely many $n,\beta$ correspond to nonzero blocks.
\end{proposition}
Again, for the Heisenberg algebra the answer is well known and can be found, for instance, in \cite{Avetisyan2016}.

We now proceed to the description of quadratic Casimir elements of the Jordanable almost Abelian Lie algebra $\aA(\aleph)=\F e_0\rtimes\mathbf{V}$. Following Section 6 of \cite{Avetisyan2016} we denote by $\mathfrak{h}:\aA(\aleph)\to\mathrm{U}(\aA(\aleph))$ the embedding of the Lie algebra into its universal enveloping algebra, and choose a basis $\{e_i\}_{i\in\dim_\F\aleph}$ in $\mathbf{V}$. Note the difference in the meaning of the symbol $\aleph$ between the preset paper and \cite{Avetisyan2016}. What was $\aleph$ in \cite{Avetisyan2016} now has become $\dim_\F\aleph$, where we identify the cardinality of a set with the set itself. One particular basis can be constructed starting from the $\Fp$-bases $\{e_\alpha^m(p,n)\}$ in each $\Fp^n$ as above. In order to build an $\F$-basis in $\Fp^n$ we generate elements $e_\alpha^{m,k}(p,n)=x_p^ke_\alpha^m(p,n)$, $k=1,\ldots,\deg p-1$. This corresponds to the labelling $i=(p,n,\alpha,m,k)$. If we let $x_i=\mathfrak{h}(e_i)$ and $\ad x_i=\mathfrak{h}(\ad_{e_0}e_i)$ then since $\ad_{e_0}=\operatorname{J}(\aleph)$ in the basis $\{e_i\}$ we have that in the basis $\{x_i\}$ of $\mathfrak{h}(\mathbf{V})$ the operator $\ad=\operatorname{J}(\aleph)$ as well.

A quadratic form $Q(x)$ on $\mathfrak{h}(\mathbf{V})$ can be written as
$$
Q(x)=\sum_{i,j\in\dim_\F\aleph}\operatorname{A}_{i;j}x_ix_j,
$$
where $\operatorname{A}\in\End_\F(\mathfrak{h}(\mathbf{V}))$ is symmetric in the basis $\{e_i\}$ and has finitely many nonzero entries in every row (it always has finitely many non-zero entries in every column as long as it is a linear operator). Formally we can also write $Q(x)=x^\top\operatorname{A}x$.
\begin{proposition}\label{Cas2aA} Quadratic Casimir elements $Q(x)\in\mathcal{Z}(\mathrm{U}(\aA(\aleph)))$ of the Jordanable almost Abelian Lie algebra $\aA(\aleph)$ are $x^\top\operatorname{A}x$, where the operator $\operatorname{A}\in\End_\F(\mathfrak{h}(\mathbf{V}))$ is symmetric, has finitely many nonzero entries in every column, and is of the form
$$
\operatorname{A}_{q,n,\beta;p,m,\alpha}=\delta_{p,(-1)\star q}(\operatorname{P}_n\otimes\operatorname{W}^\epsilon_q)(\operatorname{V}_n(-1)\otimes\operatorname{V}_{\deg q}(-1))\operatorname{R}_{q,n,\beta;p,m,\alpha},
$$
$$
(\operatorname{P}_n\otimes\operatorname{W}^\epsilon_q)=
$$

$$
\begin{pmatrix}
0 & \ldots & 0 & \begin{matrix}
0 & 0 & \ldots & 0 & 1\\
0 & 0 & \ldots & 1 & \epsilon\mu_1\\
0 & 0 & \ldots & \epsilon\mu_1 & \mu_2\\
\ldots & \ldots & \ldots & \ldots & \ldots\\
1 & \epsilon\mu_1 & \ldots & \mu_{d-2} & \mu_{d-1}
\end{matrix}\\
0 & \ldots & \begin{matrix}
0 & 0 & \ldots & 0 & 1\\
0 & 0 & \ldots & 1 & \epsilon\mu_1\\
0 & 0 & \ldots & \epsilon\mu_1 & \mu_2\\
\ldots & \ldots & \ldots & \ldots & \ldots\\
1 & \epsilon\mu_1 & \ldots & \mu_{d-2} & \mu_{d-1}
\end{matrix} & 0\\
\ldots & \ldots & \ldots & \ldots\\
\begin{matrix}
0 & 0 & \ldots & 0 & 1\\
0 & 0 & \ldots & 1 & \epsilon\mu_1\\
0 & 0 & \ldots & \epsilon\mu_1 & \mu_2\\
\ldots & \ldots & \ldots & \ldots & \ldots\\
1 & \epsilon\mu_1 & \ldots & \mu_{d-2} & \mu_{d-1}
\end{matrix} & \ldots & 0 & 0
\end{pmatrix}
$$

$$
\left(\operatorname{V}_n(-1)\otimes\operatorname{V}_{\deg q}(-1)\right)=
$$

$$
\begin{pmatrix}
\begin{matrix}
1 & 0 & \ldots & 0\\
0 & -1 & \ldots & 0\\
\ldots & \ldots & \ldots & \ldots\\
0 & 0 & \ldots & (-1)^{d-1}
\end{matrix} & 0 & \ldots & 0\\
0 & \begin{matrix}
-1 & 0 & \ldots & 0\\
0 & 1 & \ldots & 0\\
\ldots & \ldots & \ldots & \ldots\\
0 & 0 & \ldots & (-1)^{d-2}
\end{matrix} & \ldots & 0\\
\ldots & \ldots & \ldots & \ldots\\
0 & 0 & \ldots & \begin{matrix}
(-1)^{1-n} & 0 & \ldots & 0\\
0 & (-1)^{2-n} & \ldots & 0\\
\ldots & \ldots & \ldots & \ldots\\
0 & 0 & \ldots & (-1)^{d-n}
\end{matrix}
\end{pmatrix},
$$

$$
\operatorname{R}_{p;n,\beta;m,\alpha}=\begin{pmatrix}
0 & \ldots & 0 & r_1 & r_2 & \ldots & r_l\\
0 & \ldots & 0 & 0 & r_1 & \ldots & r_{l-1}\\
\ldots & \ldots & \ldots & \ldots & \ldots & \ldots & \ldots\\
0 & \ldots & 0 & 0 & 0 & \ldots & r_1
\end{pmatrix}\in\End_\Fp(\Fp^m,\Fp^n)\quad\mbox{if}\quad m\ge n,
$$

$$
\operatorname{R}_{p;n,\beta;m,\alpha}=\begin{pmatrix}
r_1 & r_2 & \ldots & r_l\\
0 & r_1 & \ldots & r_{l-1}\\
\ldots & \ldots & \ldots & \ldots\\
0 & 0 & \ldots & r_1\\
0 & 0 & \ldots & 0\\
\ldots & \ldots & \ldots & \ldots\\
0 & 0 & \ldots & 0
\end{pmatrix}\in\End_\Fp(\Fp^m,\Fp^n)\quad\mbox{if}\quad m<n,
$$

$$
l=\min\{m,n\},\quad d=\deg q,\quad r_1,\ldots,r_l\in\Fp,
$$
and $\epsilon$ is per (\ref{epsilonDef}). For every $q\in\sigma_\F$ the numbers $\mu_1,\ldots,\mu_{d-1}$ are defined as in (\ref{mundef}).
\end{proposition}
\begin{proof} By Proposition 18 of \cite{Avetisyan2016} we know that a quadratic form $Q(x)$ is a quadratic Casimir if and only if
$$
\ad Q(x)=x^\top\ad^\top\operatorname{A}x+x^\top\operatorname{A}\ad x=x^\top(\operatorname{A}\ad+\ad^\top\operatorname{A})x=0,
$$
that is,
$$
\operatorname{A}\ad+\ad^\top\operatorname{A}=\operatorname{A}\operatorname{J}(\aleph)+\operatorname{J}(\aleph)^\top\operatorname{A}=0.
$$
By Proposition \ref{TTtransCommProp} and Proposition \ref{lambdaCommProp} we know that this is equivalent to
$$
\operatorname{A}=\operatorname{W}(\aleph)\operatorname{V}(-1;\aleph)\operatorname{C},
$$
where $\operatorname{C}\in\mathrm{C}_\F(\operatorname{J}(\aleph),\operatorname{J}((-1)\star\aleph))$. Together with Lemma \ref{NCommutantLem} this yields the assertion. $\Box$ 
\end{proof}





\section{Examples}\label{Examples}
\subsection{Bianchi VI$_1$ algebra}

Let $\F=\mathbb{R}$ and $p,q\in\sigma_\F$,
$$
p(X)=X-1,\quad q(X)=X+1,
$$
i.e., $x_p=1$ and $x_q=-1$. Consider the multiplicity function $\aleph=(1\times p^1,1\times q^1)$ with notation from (\ref{AlephShort}), or even $\aleph=(1\times 1^1,1\times(-1)^1)$ if no confusion may arise. Bianchi VI$_1$ is the Jordanable almost Abelian Lie algebra
$$
\mathrm{Bi}(\mathrm{VI}_0)=\aA(\aleph)=\aA(1\times p^1,1\times q^1)=\mathbb{R}e_0\rtimes\mathbb{R}^2,
$$
with
$$
\ad_{e_0}=\operatorname{J}(\aleph)=\operatorname{J}(p,1)\oplus\operatorname{J}(q,1)=\begin{pmatrix}
1 & 0\\
0 & -1
\end{pmatrix}.
$$
We set up the adapted basis as in (\ref{AdaptBasis}),
$$
\mathbb{R}^2=\mathbb{R}e^1_1(p,1)\oplus\mathbb{R}e^1_1(q,1),\quad e^1_1(p,1)=(1,0)^\top,\quad e^1_1(q,1)=(0,1)^\top.
$$
From (\ref{ZLBasis}) we see that the centre of the algebra is
$$
\mathcal{Z}(\aA(1\times p^1,1\times q^1))=\ker\ad_{e_0}=0,
$$
while from (\ref{LkBasis}) we read that the lower central series is
$$
\aA(1\times p^1,1\times q^1)_{(k)}=\mathbb{R}e^1_1(p,1)\oplus\mathbb{R}e^1_1(q,1)=\mathbb{R}^2.
$$
Formula (\ref{LDecompBasis}) tells us that $\aA(1\times p^1,1\times q^1)$ is indecomposable.

Let us now investigate the invariant (proper) subspaces $\mathbf{W}$ of $\ad_{e_0}$ which we do with the help of Corollary \ref{InvSubspCorr}. Since $\aleph(.,n)=0$ for $n>1$ and $\aleph(p,1)=\aleph(q,1)=1$, in formula (\ref{InvSubspTransform}) we have $k=1$, $\alpha=1$ and $l=1$, so that for all $n\in\mathbb{N}$ and $m=1,\ldots, n$ we get
$$
\eta_\beta^m(p,n)=\mu_p(n,\beta;1,1,m-1)e_1^1(p,1),\quad\eta_\beta^m(q,n)=\mu_q(n,\beta;1,1,m-1)e_1^1(q,1).
$$
Corollary states that for every $\beta\in\beth(p,n)$ there exists an $\alpha\in\aleph(p,\bar n)$ such that $\bar n\ge n$ and $\mu_p(n,\beta;n,\alpha,0)\neq0$. Since $\not\exists\alpha\in\aleph(p,n)$ and $\not\exists\alpha\in\aleph(q,n)$ for $n>1$ we know that $\beth(p,n)=0$ and $\beth(q,n)=0$ for $n>1$. Therefore only $n=1$ and $m=1$ are attained,
$$
\eta_\beta^1(p,1)=\mu_p(1,\beta;1,1,0)e_1^1(p,1),\quad\eta_\beta^1(q,1)=\mu_q(1,\beta;1,1,0)e_1^1(q,1).
$$
Further, since the sum in formula (\ref{InvSubspForm}) must be direct, the vectors $\eta_\beta^1(p,1)$ must be linearly independent for different $\beta\in\beth(p,1)$, which forces $\beth(p,1)=1$. Similarly, $\beth(q,1)=1$. Thus we are left with only two possibilites,
$$
\mu_q(1,1;1,1,0)=0\quad\mbox{and}\quad\mathbf{W}=\mathbb{R}e_1^1(p,1)\oplus0
$$
or
$$
\mu_p(1,1;1,1,0)=0\quad\mbox{and}\quad\mathbf{W}=0\oplus\mathbb{R}e_1^1(q,1).
$$
This result was immediately obvious from the form of the matrix $\ad_{e_0}$, but our discussion above aimed at demonstrating how to use Corollary \ref{InvSubspCorr} rather than at the result as such.

We now proceed to study the Lie subalgebras and ideals of $\aA(1\times p^1,1\times q^1)$, which we do with the help of Proposition 4 in \cite{Avetisyan2016}. The (proper) subalgebras are of one of the following forms:
\begin{itemize}
\item Abelian Lie subalgebras of the form $\mathbf{L}=\mathbf{W}\subsetneq\mathbb{R}^2$

\item Abelian Lie subalgebras of the form $\mathbf{L}=\mathbb{R}(e_0+v_0)$ for $v_0\in\mathbb{R}^2$, since $\ker\ad_{e_0}=0$

\item Almost Abelian Lie subalgebras of the form $\mathbf{L}=\mathbb{R}(e_0+v_0)\rtimes\mathbf{W}$ with $v_0\in\mathbb{R}^2$ and $\mathbf{W}=\mathbb{R}e_1^1(p,1)\oplus0$ or $\mathbf{W}=0\oplus\mathbb{R}e_1^1(q,1)$. The Lie brackets in these subalgebras are $[e_0+v_0,e^1_1(p,1)]=e^1_1(p,1)$ and $[e_0+v_0,e^1_1(q,1)]=-e^1_1(q,1)$, and both subalgebras are isomorphic to the 2-dimensional almost Abelian Lie algebra $\mathbf{ax+b}=\aA(1\times1^1)$
\end{itemize}
The proper ideals of $\aA(1\times p^1,1\times q^1)$ are among the following:
\begin{itemize}
\item Abelian ideals $\mathbb{R}e_1^1(p,1)\oplus0$, $0\oplus\mathbb{R}e_1^1(q,1)$ and $\mathbb{R}^2$
\end{itemize}
The other two possibilities in Proposition 4 in \cite{Avetisyan2016} offer $\mathbf{L}=\mathbb{R}(e_0+v_0)\rtimes\mathbf{W}$ with $\ad_{e_0}\mathbb{R}^2\subset\mathbf{W}$. But since in our case $\ad_{e_0}\mathbb{R}^2=\mathbb{R}^2$ we obtain $\mathbf{L}=\mathbb{R}(e_0+v_0)\rtimes\mathbb{R}^2$, which is not a proper ideal.

Proceeding to automorphisms we first note that $\Dil(\aleph)=\{1,-1\}$ where
$$
(-1)\star p=q,\quad(-1)\star q=p.
$$
By Proposition \ref{AutaA} we establish that
$$
\Aut(\aA(1\times p^1,1\times q^1))=\left\{\begin{pmatrix}
1 & 0 & 0\\
\gamma_1 & \Delta_{p,1,1;p,1,1} & 0\\
\gamma_2 & 0 & \Delta_{q,1,1;q,1,1}
\end{pmatrix},\begin{pmatrix}
-1 & 0 & 0\\
\gamma_1 & 0 & \Delta_{p,1,1;q,1,1}\\
\gamma_2 & \Delta_{q,1,1;p,1,1} & 0
\end{pmatrix}\right\},
$$
$$
\gamma_1,\,\gamma_2,\,\Delta_{p,1,1;p,1,1},\,\Delta_{q,1,1;q,1,1},\,\Delta_{p,1,1;q,1,1},\,\Delta_{q,1,1;p,1,1}\in\mathbb{R},
$$
$$
\Delta_{p,1,1;p,1,1}\Delta_{q,1,1;q,1,1}\neq0\quad\Delta_{p,1,1;q,1,1}\Delta_{q,1,1;p,1,1}\neq0.
$$
For derivations we use Proposition \ref{DeraANonnil},
$$
\Der(\aA(1\times p^1,1\times q^1))=\left\{\begin{pmatrix}
0 & 0 & 0\\
\gamma_1 & \Delta_{p,1,1;p,1,1} & 0\\
\gamma_2 & 0 & \Delta_{q,1,1;q,1,1}
\end{pmatrix}\right\},
$$
$$
\gamma_1,\,\gamma_2,\,\Delta_{p,1,1;p,1,1},\,\Delta_{q,1,1;p,1,1}\in\mathbb{R}.
$$
Finally, quadratic Casimir elements are found from Proposition \ref{Cas2aA}. It follows that the matrix $\operatorname{A}$ has the form
$$
\operatorname{A}=\begin{pmatrix}
0 & a\\
b & 0
\end{pmatrix},\quad a,b\in\mathbb{R},
$$
and the symmetry forces $a=b$. If we denote $x=\mathfrak{h}(e_1^1(p,1))$ and $y=\mathfrak{h}(e_1^1(q,1))$ then the quadratic Casimir is
$$
Q=axy,\quad a\in\mathbb{R}.
$$

\subsection{Mautner algebra}

Let $\F=\mathbb{R}$ and $p,q\in\sigma_\F$,
$$
p(X)=X^2+1,\quad q(X)=X+4\pi^2.
$$
On this occasion let us adopt the traditional representation (\ref{xpMatrixFormR}),
$$
x_p=\begin{pmatrix}
0 & -1\\
1 & 0
\end{pmatrix},\quad x_q=\begin{pmatrix}
0 & -2\pi\\
2\pi & 0
\end{pmatrix},
$$
so that $\epsilon=0$. Consider the multiplicity function $\aleph=(1\times p^1,1\times q^1)$. The Mautner algebra (the Lie algebra of the prominent Mautner group) is the Jordanable almost Abelian Lie algebra
$$
\aA(\aleph)=\aA(1\times p^1,1\times q^1)=\mathbb{R}e_0\rtimes\mathbb{R}^4,
$$
with
$$
\ad_{e_0}=\operatorname{J}(\aleph)=\operatorname{J}(p,1)\oplus\operatorname{J}(q,1)=\begin{pmatrix}
0 & -1 & 0 & 0\\
1 & 0 & 0 & 0\\
0 & 0 & 0 & -2\pi\\
0 & 0 & -2\pi & 0
\end{pmatrix}.
$$
In both cases
$$
\Fp=\mathbb{R}(x_p)=\mathbb{R}(x_q)=\left\{\begin{pmatrix}
a & -b\\
b & a
\end{pmatrix}\,\vline\quad a,b\in\mathbb{R}\right\}\simeq\mathbb{C},
$$
and we will denote this representation of the field extension simply by $\mathbb{C}$. Had we chosen (\ref{xpMatrixForm}) instead we would need to distinguish between $\mathbb{R}(x_p)=\mathbb{C}_p$ and $\mathbb{R}(x_q)=\mathbb{C}_q$ which are both isomorphic to $\mathbb{C}$ but have different representations as $\mathbb{R}$-matrices. We set up the adapted basis as in (\ref{AdaptBasis}),
$$
\mathbb{R}^4=\mathbb{C}e^1_1(p,1)\oplus\mathbb{C}e^1_1(q,1),\quad e^1_1(p,1)=(1,0,0,0)^\top,\quad e^1_1(q,1)=(0,0,1,0)^\top.
$$
From (\ref{ZLBasis}) we see that the centre of the algebra is
$$
\mathcal{Z}(\aA(1\times p^1,1\times q^1))=\ker\ad_{e_0}=0,
$$
while from (\ref{LkBasis}) we read that the lower central series is
$$
\aA(1\times p^1,1\times q^1)_{(k)}=\mathbb{C}e^1_1(p,1)\oplus\mathbb{C}e^1_1(q,1)=\mathbb{R}^4.
$$
Formula (\ref{LDecompBasis}) tells us that $\aA(1\times p^1,1\times q^1)$ is indecomposable.

Let us now investigate the invariant (proper) subspaces $\mathbf{W}$ of $\ad_{e_0}$ which we do with the help of Corollary \ref{InvSubspCorr}. Since $\aleph(.,n)=0$ for $n>1$ and $\aleph(p,1)=\aleph(q,1)=1$, in formula (\ref{InvSubspTransform}) we have $k=1$, $\alpha=1$ and $l=1$, so that for all $n\in\mathbb{N}$ and $m=1,\ldots, n$ we get
$$
\eta_\beta^m(p,n)=\mu_p(n,\beta;1,1,m-1)e_1^1(p,1),\quad\eta_\beta^m(q,n)=\mu_q(n,\beta;1,1,m-1)e_1^1(q,1).
$$
Corollary states that for every $\beta\in\beth(p,n)$ there exists an $\alpha\in\aleph(p,\bar n)$ such that $\bar n\ge n$ and $\mu_p(n,\beta;n,\alpha,0)\neq0$. Since $\not\exists\alpha\in\aleph(p,n)$ and $\not\exists\alpha\in\aleph(q,n)$ for $n>1$ we know that $\beth(p,n)=0$ and $\beth(q,n)=0$ for $n>1$. Therefore only $n=1$ and $m=1$ are attained,
$$
\eta_\beta^1(p,1)=\mu_p(1,\beta;1,1,0)e_1^1(p,1),\quad\eta_\beta^1(q,1)=\mu_q(1,\beta;1,1,0)e_1^1(q,1).
$$
Further, since the sum in formula (\ref{InvSubspForm}) must be direct, the vectors $\eta_\beta^1(p,1)$ must be linearly independent for different $\beta\in\beth(p,1)$, which forces $\beth(p,1)=1$. Similarly, $\beth(q,1)=1$. Thus we are left with only two possibilites,
$$
\mu_q(1,1;1,1,0)=0\quad\mbox{and}\quad\mathbf{W}=\mathbb{C}e_1^1(p,1)\oplus0
$$
or
$$
\mu_p(1,1;1,1,0)=0\quad\mbox{and}\quad\mathbf{W}=0\oplus\mathbb{C}e_1^1(q,1).
$$

We now proceed to study the Lie subalgebras and ideals of $\aA(1\times p^1,1\times q^1)$, which we do with the help of Proposition 4 in \cite{Avetisyan2016}. The (proper) subalgebras are of one of the following forms:
\begin{itemize}
\item Abelian Lie subalgebras of the form $\mathbf{L}=\mathbf{W}\subsetneq\mathbb{R}^4$

\item Abelian Lie subalgebras of the form $\mathbf{L}=\mathbb{R}(e_0+v_0)$ for $v_0\in\mathbb{R}^4$, since $\ker\ad_{e_0}=0$

\item Almost Abelian Lie subalgebras of the form $\mathbf{L}=\mathbb{R}(e_0+v_0)\rtimes\mathbf{W}$ with $v_0\in\mathbb{R}^4$ and $\mathbf{W}=\mathbb{C}e_1^1(p,1)\oplus0$ or $\mathbf{W}=0\oplus\mathbb{C}e_1^1(q,1)$. The Lie brackets in these subalgebras are $[e_0+v_0,e^1_1(p,1)]=x_pe^1_1(p,1)$ and $[e_0+v_0,e^1_1(q,1)]=x_qe^1_1(q,1)$, and both subalgebras are isomorphic to the 3-dimensional almost Abelian Lie algebra $\mathrm{Bi}(\mathrm{VII}_0)=\mathbf{E}(2)=\aA(1\times p^1)$.
\end{itemize}
The proper ideals of $\aA(1\times p^1,1\times q^1)$ are among the following:
\begin{itemize}
\item Abelian ideals $\mathbb{C}e_1^1(p,1)\oplus0$, $0\oplus\mathbb{C}e_1^1(q,1)$ and $\mathbb{R}^4$
\end{itemize}
The other two possibilities in Proposition 4 in \cite{Avetisyan2016} offer $\mathbf{L}=\mathbb{R}(e_0+v_0)\rtimes\mathbf{W}$ with $\ad_{e_0}\mathbb{R}^4\subset\mathbf{W}$. But since in our case $\ad_{e_0}\mathbb{R}^4=\mathbb{R}^4$ we obtain $\mathbf{L}=\mathbb{R}(e_0+v_0)\rtimes\mathbb{R}^4$, which is not a proper ideal.

Proceeding to automorphisms we first note that $\Dil(\aleph)=\{1,-1\}$ where
$$
(-1)\star p=p,\quad(-1)\star q=q.
$$
By Proposition \ref{AutaA} we establish that
$$
\Aut(\aA(1\times p^1,1\times q^1))=\left\{\begin{pmatrix}
\pm1 & 0 & 0 & 0 & 0\\
\gamma_1 & \Delta_{p,1,1;p,1,1}^r & -\Delta_{p,1,1;p,1,1}^i & 0 & 0\\
\gamma_2 & \Delta_{p,1,1;p,1,1}^i & \Delta_{p,1,1;p,1,1}^r & 0 & 0\\
\gamma_3 & 0 & 0 & \Delta_{q,1,1;q,1,1}^r & -\Delta_{q,1,1;q,1,1}^i\\
\gamma_4 & 0 & 0 & \Delta_{q,1,1;q,1,1}^i & \Delta_{q,1,1;q,1,1}^r\\
\end{pmatrix}\right\},
$$

$$
\gamma_1,\,\gamma_2,\,\gamma_3,\,\gamma_4\in\mathbb{R},
$$

$$
\begin{pmatrix}
\Delta_{p,1,1;p,1,1}^r & -\Delta_{p,1,1;p,1,1}^i\\
\Delta_{p,1,1;p,1,1}^i & \Delta_{p,1,1;p,1,1}^r
\end{pmatrix},\begin{pmatrix}
\Delta_{q,1,1;p,1,1}^r & -\Delta_{q,1,1;p,1,1}^i\\
\Delta_{q,1,1;p,1,1}^i & \Delta_{q,1,1;p,1,1}^r
\end{pmatrix}\in\mathbb{C},
$$

$$
(\Delta_{p,1,1;p,1,1}^r)^2+(\Delta_{p,1,1;p,1,1}^i)^2+(\Delta_{q,1,1;q,1,1}^r)^2+(\Delta_{q,1,1;q,1,1}^i)^2>0.
$$
For derivations we use Proposition \ref{DeraANonnil},
$$
\Der(\aA(1\times p^1,1\times q^1))=\left\{\begin{pmatrix}
0 & 0 & 0 & 0 & 0\\
\gamma_1 & \Delta_{p,1,1;p,1,1}^r & -\Delta_{p,1,1;p,1,1}^i & 0 & 0\\
\gamma_2 & \Delta_{p,1,1;p,1,1}^i & \Delta_{p,1,1;p,1,1}^r & 0 & 0\\
\gamma_3 & 0 & 0 & \Delta_{q,1,1;q,1,1}^r & -\Delta_{q,1,1;q,1,1}^i\\
\gamma_4 & 0 & 0 & \Delta_{q,1,1;q,1,1}^i & \Delta_{q,1,1;q,1,1}^r\\
\end{pmatrix}\right\},
$$
$$
\begin{pmatrix}
\Delta_{p,1,1;p,1,1}^r & -\Delta_{p,1,1;p,1,1}^i\\
\Delta_{p,1,1;p,1,1}^i & \Delta_{p,1,1;p,1,1}^r
\end{pmatrix},\begin{pmatrix}
\Delta_{q,1,1;p,1,1}^r & -\Delta_{q,1,1;p,1,1}^i\\
\Delta_{q,1,1;p,1,1}^i & \Delta_{q,1,1;p,1,1}^r
\end{pmatrix}\in\mathbb{C}.
$$
Finally, quadratic Casimir elements are found from Proposition \ref{Cas2aA}. It follows that
$$
(\operatorname{P}_1\otimes\operatorname{W}_p^0)(\operatorname{V}_1(-1)\otimes\operatorname{V}_2(-1))=(\operatorname{P}_1\otimes\operatorname{W}_q^0)(\operatorname{V}_1(-1)\otimes\operatorname{V}_2(-1))=\begin{pmatrix}
0 & -1\\
1 & 0
\end{pmatrix},
$$
and the matrix $\operatorname{A}$ has the form
$$
\operatorname{A}=\begin{pmatrix}
a & -b & 0 & 0\\
b & a & 0 &0\\
0 & 0 & c & -d\\
0 & 0 & d & c
\end{pmatrix},\quad a,b,c,d\in\mathbb{R},
$$
while symmetry requires $b=d=0$. If we denote 
$$
x=\mathfrak{h}(e_1^1(p,1)),\quad y=\mathfrak{h}(x_pe_1^1(p,1)),\quad z=\mathfrak{h}(e_1^1(q,1)),\quad w=\mathfrak{h}(x_qe_1^1(q,1)),
$$
then the quadratic Casimir is
$$
Q=a(x^2+y^2)+c(z^2+w^2),\quad a,c\in\mathbb{R}.
$$

\subsection{Algebra $\aA_\mathbb{Q}(1\times\sqrt[3]{2}\,^2)\oplus\mathbb{Q}$}

Let $\F=\mathbb{Q}$ and $p,q\in\sigma_\F$,
$$
p(X)=X^3-2,\quad q(X)=X.
$$
According to (\ref{xpMatrixForm}),
$$
x_p=\begin{pmatrix}
0 & 0 & 2\\
1 & 0 & 0\\
0 & 1 & 0
\end{pmatrix},\quad x_q=0,
$$
and $\epsilon=1$. Consider the multiplicity function $\aleph=(1\times p^2,1\times q^1)$ and the Jordanable almost Abelian Lie algebra
$$
\aA_\mathbb{Q}(\aleph)=\aA(\aleph)=\aA(1\times p^2,1\times q^1)=\mathbb{Q}e_0\rtimes\mathbb{Q}^7,
$$
with
$$
\ad_{e_0}=\operatorname{J}(\aleph)=\operatorname{J}(p,2)\oplus\operatorname{J}(q,1)=\begin{pmatrix}[ccc|ccc|c]
0 & 0 & 2 & 1 & 0 & 0 & 0\\
1 & 0 & 0 & 0 & 1 & 0 & 0\\
0 & 1 & 0 & 0 & 0 & 1 & 0\\\hline
0 & 0 & 0 & 0 & 0 & 2 & 0\\
0 & 0 & 0 & 1 & 0 & 0 & 0\\
0 & 0 & 0 & 0 & 1 & 0 & 0\\\hline
0 & 0 & 0 & 0 & 0 & 0 & 0
\end{pmatrix}.
$$
In this case we have
$$
\Fp=\mathbb{Q}(x_p)=\left\{\begin{pmatrix}
a & 2c & 2b\\
b & a & 2c\\
c & b & a
\end{pmatrix}\,\vline\quad a,b,c\in\mathbb{Q}\right\}\simeq\mathbb{Q}(\sqrt[3]{2}),
$$
whereas $\mathbb{Q}(x_q)=\mathbb{Q}$. We set up the adapted basis as in (\ref{AdaptBasis}),
$$
\mathbb{Q}^7=\Fp e^1_1(p,2)\oplus\Fp e^2_1(p,2)\oplus\mathbb{Q}e^1_1(q,1),
$$
$$
e^1_1(p,2)=(1,0,0,0,0,0,0)^\top,\quad e^2_1(p,2)=(0,0,0,1,0,0,0)^\top,\quad e^1_1(q,1)=(0,0,0,0,0,0,1)^\top.
$$
From (\ref{ZLBasis}) we see that the centre of the algebra is
$$
\mathcal{Z}(\aA(1\times p^2,1\times q^1))=\ker\ad_{e_0}=\mathbb{Q}e^1_1(q,1),
$$
while from (\ref{LkBasis}) we read that the lower central series is
$$
\aA(1\times p^2,1\times q^1)_{(k)}=\Fp e^1_1(p,2)\oplus\Fp e^2_1(p,2)\oplus0.
$$
Formula (\ref{LDecompBasis}) tells us that 
$$
\aA(1\times p^2,1\times q^1)=\aA(1\times p^2)\oplus\mathbb{Q},
$$
hence the title of this subsection. 

Let us now investigate the invariant (proper) subspaces $\mathbf{W}$ of $\ad_{e_0}$ which we do with the help of Corollary \ref{InvSubspCorr}. First consider $p$. Since $\aleph(p,2)=1$ and $\aleph(p,n)=0$ for $n\neq2$ we have $k=2$ and $\alpha=1$ only in formula (\ref{InvSubspTransform}). Moreover, the corollary states that for every $\beta\in\beth(p,n)$ there exists an $\alpha\in\aleph(p,\bar n)$ such that $\bar n\ge n$ and $\mu_p(n,\beta;n,\alpha,0)\neq0$. Since $\not\exists\alpha\in\aleph(p,n)$ for $n\neq2$ we find that $n\in\{1,2\}$. For $n=1$ we get
$$
\eta_\beta^1(p,1)=\mu_p(1,\beta;2,1,0)\,e_1^1(p,2),\quad\forall\beta\in\beth(p,1).
$$
For $n=2$ we obtain for $\forall\beta\in\beth(p,2)$
$$
\eta_\beta^1(p,2)=\mu_p(2,\beta;2,1,0)\,e_1^1(p,2),\quad\eta_\beta^2(p,2)=\mu_p(2,\beta;2,1,1)\,e_1^1(p,2)+\mu_p(2,\beta;2,1,0)\,e_2^1(p,2).
$$
Now since the sum in formula (\ref{InvSubspForm}) must be direct, we see that either $\mu_p(1,\beta;2,1,0)=0$ and $\beth(p,1)=0$ or $\mu_p(2,\beta;2,1,0)=0$ and $\beth(p,2)=0$. Following similar reasoning, from $\aleph(q,1)=1$ and $\aleph(q,n)=0$ for $n>1$ we establish that $k=1$, $\alpha=1$ and $n=1$ in (\ref{InvSubspTransform}). Therefore
$$
\eta_\beta^1(q,1)=\mu_p(1,\beta;1,1,0)\,e_1^1(q,1),\quad\forall\beta\in\beth(q,1).
$$
Again arguments of linear independence show that $\beth(q,1)\le1$. To conclude, we have the following possibilities for an invariant proper subspace:
\begin{itemize}

\item $\beth=(1\times p^2)$ and $\mathbf{W}=\Fp e^1_1(p,2)\oplus\Fp e^2_1(p,2)$

\item $\beth=(1\times p^1)$ and $\mathbf{W}=\Fp e^1_1(p,2)$

\item $\beth=(1\times q^1)$ and $\mathbf{W}=\mathbb{Q}e^1_1(q,1)$

\item $\beth=(1\times p^1,1\times q^1)$ and $\mathbf{W}=\Fp e^1_1(p,2)\oplus\mathbb{Q}e^1_1(q,1)$

\end{itemize}

We now proceed to study the Lie subalgebras and ideals of $\aA(1\times p^2,1\times q^1)$, which we do with the help of Proposition 4 in \cite{Avetisyan2016}. The (proper) subalgebras are of one of the following forms:
\begin{itemize}
\item Abelian Lie subalgebras of the form $\mathbf{L}=\mathbf{W}\subsetneq\mathbb{Q}^7$

\item Abelian Lie subalgebras of the form $\mathbf{L}=\mathbb{Q}(e_0+v_0)$ or $\mathbf{L}=\mathbb{Q}(e_0+v_0)\oplus\mathbb{Q}e_1^1(q,1)$ for $v_0\in\mathbb{Q}^7$

\item Almost Abelian Lie subalgebras of the form $\mathbf{L}=\mathbb{Q}(e_0+v_0)\rtimes\mathbf{W}$ with $v_0\in\mathbb{Q}^7$ and either of the following: a) $\mathbf{W}=\Fp e_1^1(p,2)$ so that $\mathbf{L}\simeq\aA(1\times p^1)$, b) $\mathbf{W}=\Fp e_1^1(p,2)\oplus\Fp e_1^2(p,2)$ so that $\mathbf{L}\simeq\aA(1\times p^2)$, c) $\mathbf{W}=\Fp e_1^1(p,2)\oplus\mathbb{Q}e_1^1(q,1)$ so that $\mathbf{L}\simeq\aA(1\times p^1)\oplus\mathbb{Q}$.

\end{itemize}
The proper ideals of $\aA(1\times p^2,1\times q^1)$ are among the following:
\begin{itemize}

\item Abelian ideals $\mathbf{L}=\mathbf{W}\subseteq\mathbb{Q}^7$ that are invariant subspaces, proper as above or improper

\item Almost Abelian ideal $\mathbf{L}=\mathbb{Q}\rtimes(\Fp e_1^1(p,2)\oplus\Fp e_1^2(p,2))\simeq\aA(1\times p^2)$

\end{itemize}
The remaining possibility in Proposition 4 in \cite{Avetisyan2016} requires $\ad_{e_0}\mathbb{Q}^7\subset\ker\ad_{e_0}$ which does not hold (see Remark 3 in \cite{Avetisyan2016}).

Let us now proceed to the automorphisms. Our algebra has a decomposition
$$
\aA(1\times p^2,1\times q^1)=\mathbf{L}_0\oplus\mathbf{W},
$$
$$
\mathbf{L}_0=\mathbb{Q}e_0\rtimes(\Fp e_1^1(p,2)\oplus\Fp e_1^2(p,2))=\aA(1\times p^2),\quad\mathbf{W}=\mathbb{Q}e_1^1(q,1).
$$
Let us deal with $\mathbf{L}_0$ first. Since $\Dil(\aleph)=\{1\}$, by Proposition \ref{AutaA} we establish that
$$
\Aut(\aA(1\times p^2))=\left\{\begin{pmatrix}[c|ccc|ccc]
1 & 0 & 0 & 0 & 0 & 0 & 0\\\hline
\gamma_1 & \Delta_a & 2\Delta_c & 2\Delta_b & \Delta_d & 2\Delta_f & 2\Delta_e\\
\gamma_2 & \Delta_b & \Delta_a & 2\Delta_c & \Delta_e & \Delta_d & 2\Delta_f\\
\gamma_3 & \Delta_c & \Delta_b & \Delta_a & \Delta_f & \Delta_e & \Delta_d\\\hline
\gamma_4 & 0 & 0 & 0 & \Delta_a & 2\Delta_c & 2\Delta_b\\
\gamma_5 & 0 & 0 & 0 & \Delta_b & \Delta_a & 2\Delta_c\\
\gamma_6 & 0 & 0 & 0 & \Delta_c & \Delta_b & \Delta_a
\end{pmatrix}\right\},
$$

$$
\gamma_1,\,\gamma_2,\,\gamma_3,\,\gamma_4,\,\gamma_5,\,\gamma_6\in\mathbb{Q},
$$

$$
\begin{pmatrix}
\Delta_a & 2\Delta_c & 2\Delta_b\\
\Delta_b & \Delta_a & 2\Delta_c\\
\Delta_c & \Delta_b & \Delta_a
\end{pmatrix},\begin{pmatrix}
\Delta_d & 2\Delta_f & 2\Delta_e\\
\Delta_e & \Delta_d & 2\Delta_f\\
\Delta_f & \Delta_e & \Delta_d
\end{pmatrix}\in\Fp,
$$

$$
(\Delta_a)^2+(\Delta_b)^2+(\Delta_c)^2>0.
$$
Coming back to the direct sum $\aA(1\times p^2,1\times q^1)=\aA(1\times p^2)\oplus\mathbb{Q}$, \mbox{Proposition 7} and \mbox{Proposition 8} in \cite{Avetisyan2016} yield
$$
\Aut(\aA(1\times p^2,1\times q^1))=\left\{\begin{pmatrix}
\phi_{00} & \phi_{01}\\
\phi_{10} & \phi_{11}
\end{pmatrix}\right\},
$$
$$
\phi_{00}\in\Aut(\mathbf{L}_0),\quad\phi_{01}(\mathbf{W})\in\mathcal{Z}(\mathbf{L}_0)=0,\quad(\mathbf{L}_0)_{(1)}=\Fp e_1^1(p,2)\oplus\Fp e_1^2(p,2)\subset\ker\phi_{10},\quad\ker\phi_{11}=0.
$$
Thus
$$
\Aut(\aA(1\times p^2,1\times q^1))=\left\{\begin{pmatrix}[c|ccc|ccc|c]
1 & 0 & 0 & 0 & 0 & 0 & 0 & 0\\\hline
\gamma_1 & \Delta_a & 2\Delta_c & 2\Delta_b & \Delta_d & 2\Delta_f & 2\Delta_e & 0\\
\gamma_2 & \Delta_b & \Delta_a & 2\Delta_c & \Delta_e & \Delta_d & 2\Delta_f & 0\\
\gamma_3 & \Delta_c & \Delta_b & \Delta_a & \Delta_f & \Delta_e & \Delta_d & 0\\\hline
\gamma_4 & 0 & 0 & 0 & \Delta_a & 2\Delta_c & 2\Delta_b & 0\\
\gamma_5 & 0 & 0 & 0 & \Delta_b & \Delta_a & 2\Delta_c & 0\\
\gamma_6 & 0 & 0 & 0 & \Delta_c & \Delta_b & \Delta_a & 0\\\hline
\gamma_7 & 0 & 0 & 0 & 0 & 0 & 0 & \delta
\end{pmatrix}\right\},
$$

$$
\gamma_7,\delta\in\mathbb{Q},\quad\delta\neq0.
$$
In dealing with derivations let us again start with $\mathbf{L}_0=\aA(1\times p^2)$. We use Proposition \ref{DeraANonnil},
$$
\Der(\aA(1\times p^2))=\left\{\begin{pmatrix}[c|ccc|ccc]
0 & 0 & 0 & 0 & 0 & 0 & 0\\\hline
\gamma_1 & \Delta_a & 2\Delta_c & 2\Delta_b & \Delta_d & 2\Delta_f & 2\Delta_e\\
\gamma_2 & \Delta_b & \Delta_a & 2\Delta_c & \Delta_e & \Delta_d & 2\Delta_f\\
\gamma_3 & \Delta_c & \Delta_b & \Delta_a & \Delta_f & \Delta_e & \Delta_d\\\hline
\gamma_4 & 0 & 0 & 0 & \Delta_a & 2\Delta_c & 2\Delta_b\\
\gamma_5 & 0 & 0 & 0 & \Delta_b & \Delta_a & 2\Delta_c\\
\gamma_6 & 0 & 0 & 0 & \Delta_c & \Delta_b & \Delta_a
\end{pmatrix}\right\},
$$

$$
\gamma_1,\,\gamma_2,\,\gamma_3,\,\gamma_4,\,\gamma_5,\,\gamma_6\in\mathbb{Q},
$$

$$
\begin{pmatrix}
\Delta_a & 2\Delta_c & 2\Delta_b\\
\Delta_b & \Delta_a & 2\Delta_c\\
\Delta_c & \Delta_b & \Delta_a
\end{pmatrix},\begin{pmatrix}
\Delta_d & 2\Delta_f & 2\Delta_e\\
\Delta_e & \Delta_d & 2\Delta_f\\
\Delta_f & \Delta_e & \Delta_d
\end{pmatrix}\in\Fp.
$$
For the direct sum $\aA(1\times p^2,1\times q^1)=\aA(1\times p^2)\oplus\mathbb{Q}$ we apply Proposition 12 in \cite{Avetisyan2016}, which tells us that
$$
\Der(\aA(1\times p^2,1\times q^1))=\left\{\begin{pmatrix}
\phi_{00} & \phi_{01}\\
\phi_{10} & \phi_{11}
\end{pmatrix}\right\},
$$
$$
\phi_{00}\in\Der(\mathbf{L}_0),\quad\phi_{01}(\mathbf{W})\in\mathcal{Z}(\mathbf{L}_0)=0,\quad(\mathbf{L}_0)_{(1)}=\Fp e_1^1(p,2)\oplus\Fp e_1^2(p,2)\subset\ker\phi_{10}.
$$
Thus
$$
\Der(\aA(1\times p^2,1\times q^1))=\left\{\begin{pmatrix}[c|ccc|ccc|c]
0 & 0 & 0 & 0 & 0 & 0 & 0 & 0\\\hline
\gamma_1 & \Delta_a & 2\Delta_c & 2\Delta_b & \Delta_d & 2\Delta_f & 2\Delta_e & 0\\
\gamma_2 & \Delta_b & \Delta_a & 2\Delta_c & \Delta_e & \Delta_d & 2\Delta_f & 0\\
\gamma_3 & \Delta_c & \Delta_b & \Delta_a & \Delta_f & \Delta_e & \Delta_d & 0\\\hline
\gamma_4 & 0 & 0 & 0 & \Delta_a & 2\Delta_c & 2\Delta_b & 0\\
\gamma_5 & 0 & 0 & 0 & \Delta_b & \Delta_a & 2\Delta_c & 0\\
\gamma_6 & 0 & 0 & 0 & \Delta_c & \Delta_b & \Delta_a & 0\\\hline
\gamma_7 & 0 & 0 & 0 & 0 & 0 & 0 & \delta
\end{pmatrix}\right\},
$$

$$
\gamma_7,\delta\in\mathbb{Q}.
$$
Finally, quadratic Casimir elements are found from Proposition \ref{Cas2aA}. Since $(-1)\star q=q$ and
$$
\aleph((-1)\star p,.)=0,
$$ 
we see that only the $q;q$-block contributes to a Casimir element,
$$
\operatorname{A}=\begin{pmatrix}[c|ccc|ccc|c]
0 & 0 & 0 & 0 & 0 & 0 & 0 & 0\\\hline
0 & 0 & 0 & 0 & 0 & 0 & 0 & 0\\
0 & 0 & 0 & 0 & 0 & 0 & 0 & 0\\
0 & 0 & 0 & 0 & 0 & 0 & 0 & 0\\\hline
0 & 0 & 0 & 0 & 0 & 0 & 0 & 0\\
0 & 0 & 0 & 0 & 0 & 0 & 0 & 0\\
0 & 0 & 0 & 0 & 0 & 0 & 0 & 0\\\hline
0 & 0 & 0 & 0 & 0 & 0 & 0 & a
\end{pmatrix},\quad a\in\mathbb{Q}.
$$
If we denote $w=\mathfrak{h}(e_1^1(q,1))$ then the Casimir element is
$$
Q=aw^2.
$$








\end{document}